\documentclass[a4paper,12pt]{article}

\usepackage[T2A]{fontenc}
\usepackage[utf8]{inputenc}
\usepackage[russian]{babel}
\usepackage{amsmath, amssymb, amsthm, eqnarray}
\usepackage{enumerate}
\usepackage{indentfirst}
\usepackage[colorlinks=true,linkcolor=blue,citecolor=blue]{hyperref}
\usepackage{tikz-cd}
\usetikzlibrary{cd}
\usetikzlibrary{matrix, arrows}

\usepackage{tikz}
\usepackage{tikz-3dplot} 
\usetikzlibrary{fadings}
\tikzfading[name=fade out,
            inner color=transparent!0,
            outer color=transparent!100]
\usetikzlibrary{calc,arrows,decorations.pathreplacing,fadings,3d,positioning}

\title{Собственные циклические симметрии многомерных цепных дробей.
      }
     \date{}

\author{И.\,А.\,Тлюстангелов}

\theoremstyle{definition}
\newtheorem{definition}{Определение}

\newtheorem*{notation*}{Обозначение}
\theoremstyle{remark}
\newtheorem{remark}{Замечание}
\newtheorem*{remark*}{Замечание}
\theoremstyle{plain}
\newtheorem{theorem}{Теорема}
\newtheorem*{theorem_dir*}{Теорема Дирихле}
\newtheorem{lemma}{Лемма}

\newtheorem{proposition}{Предложение}
\newtheorem{corollary}{Следствие}
\newtheorem*{statement*}{Утверждение}
\newtheorem*{corollary*}{Следствие}

\newtheorem{proof_m*}{Доказательство теоремы 1}

\DeclareMathOperator{\conv}{conv}

\DeclareMathOperator{\id}{id}

\renewcommand{\phi}{\varphi}

\renewcommand{\vec}[1]{\mathbf{#1}}

\newcommand{\R}{\mathbb{R}}
\newcommand{\Z}{\mathbb{Z}}
\newcommand{\Q}{\mathbb{Q}}

\newcommand{\cK}{\mathcal{K}}

\newcommand{\gA}{\mathfrak{A}}

\newcommand{\Sl}{\textup{SL}}
\newcommand{\Gl}{\textup{GL}}
\newcommand{\norm}{\textup{N}}
\newcommand{\trace}{\textup{Tr}}
\newcommand{\cf}{\textup{CF}}

\newcommand{\gal}{\textup{Gal}}

\begin{document}

\maketitle

\begin{abstract}
Данная работа посвящена доказательству утверждения о существовании в произвольной размерности палиндромичных цепных дробей. Кроме того, доказывается критерий наличия у алгебраической цепной дроби собственной циклической палиндромической симметрии в случае $n=4$. В качестве многомерного обобщения цепных дробей рассматриваются полиэдры Клейна. 

\end{abstract}

\section{Введение}\label{intro}

Для понятия классической цепной дроби действительного числа известно несколько обобщений, одно из которых основывается на геометрической интерпретации цепной дроби, предложенной Клейном \cite{klein}. А именно, пусть $l_1,\ldots,l_n$ --- одномерные подпространства пространства $\R^n$, линейная оболочка которых совпадает со всем $\R^n$. Тогда гиперпространства, натянутые на всевозможные $(n-1)$-наборы из этих подпространств, разбивают $\R^n$ на $2^n$ симплициальных конусов. Будем обозначать множество этих конусов через
\[ \mathcal{C}(l_1, \ldots, l_n).\]
Симплициальный конус с вершиной в начале координат $\vec{0}$ будем называть \emph{иррациональным}, если линейная оболочка любой его гиперграни не содержит целых точек, кроме начала координат  $\vec{0}$.

\begin{definition}
  Пусть  $C$ --- иррациональный конус, $C \in \mathcal{C}(l_1, \ldots, l_n)$. Выпуклая оболочка $\cK(C) = \conv(C\cap\Z^{n}\setminus\{\vec{0}\} )$ и его граница  $\partial(\cK(C))$ называются соответственно \emph{полиэдром Клейна} и \emph{парусом Клейна}, соответствующими конусу $C$. Объединение же всех $2^n$ парусов
  \[\cf(l_1, \ldots, l_n) = {\underset{C \, \in \, \mathcal{C}(l_1, \ldots, l_n)}{\bigcup}} \partial(\cK(C))\]
  называется \emph{$(n-1)$-мерной цепной дробью}.
\end{definition}

Классическая теорема Лагранжа (\cite{lagrange}, \cite{khintchine_CF}) о цепных дробях утверждает, что число $\alpha$ является квадратичной иррациональностью тогда и только тогда, когда цепная дробь числа $\alpha$ периодична начиная с некоторого момента. Геометрическая интерпретация теоремы Лагранжа основывается на том факте, что вектор $(1, \alpha)$ является собственным вектором некоторого $\Sl_{2}(\Z)$ оператора с различными вещественными собственными значениями в том и только том случае, если число  $\alpha$ --- квадратичная иррациональность (см., например \cite{german_tlyust}). Данное утверждение обобщается естественным образом на случай произвольного $n$. Напомним, что оператор из $\Gl_{n}(\Z)$ с вещественными собственными значениями, характеристический многочлен которого неприводим над $\Q$, называется \emph{гиперболическим}.

\begin{definition}
  Пусть $l_1,\ldots,l_n$ --- собственные подпространства некоторого гиперболического оператора $A\in\Gl_n(\Z)$. Тогда $(n-1)$-мерная цепная дробь $\cf(l_1,\ldots,l_n)$ называется \emph{алгебраической}. Мы будем также говорить, что эта дробь \emph{ассоциирована} с оператором $A$ и писать $\cf(A)=\cf(l_1,\ldots,l_n)$. Множество всех $(n-1)$-мерных алгебраических цепных дробей будем обозначать $\gA_{n-1}$.
\end{definition}

Упомянутое выше обобщение выглядит следующим образом (подробности см., например, в \cite{german_tlyust_2})

\begin{proposition}\label{prop:more_than_pelle_n_dim}
  Числа $1,\alpha_1,\ldots,\alpha_{n-1}$ образуют базис некоторого вполне вещественного расширения $K$ поля $\Q$ тогда и только тогда, когда вектор $(1,\alpha_1,\ldots,\alpha_{n-1})$ является собственным для некоторого гиперболического оператора $A\in\Sl_n(\Z)$.
  При этом вектора $(1,\sigma_i(\alpha_1),\ldots,\sigma_i(\alpha_{n-1}))$, $i=1,\ldots,n$, где $\sigma_1(=\id),\sigma_2,\ldots,\sigma_n$ --- все вложения $K$ в $\R$, образуют собственный базис оператора $A$.
\end{proposition}

Будем называть \emph{группой симметрий} алгебраической цепной дроби $\cf(A)=\cf(l_1,\ldots,l_n)$ множество
\[
  \textup{Sym}_{\Z}\big(\cf(A)\big)=
  \Big\{ G\in\Gl_n(\Z) \ \Big|\ G\big(\cf(A)\big)=\cf(A) \Big\}.
\]
Из соображений непрерывности ясно, что для каждого $G\in\textup{Sym}_{\Z}\big(\cf(A)\big)$ однозначно определена перестановка $\sigma_G$, такая что
\begin{equation} \label{eq:repres}
  G(l_{i})=l_{\sigma_G(i)},\quad i=1,\dots,n.
\end{equation}
И обратно, если для $G\in\Gl_n(\Z)$ существует такая перестановка $\sigma_{G}$, что выполняются соотношения \ref{eq:repres}, то $G\in\textup{Sym}_{\Z}\big(\cf(A)\big)$.

\begin{definition}
  Оператор $G\in\textup{Sym}_{\Z}\big(\cf(A)\big)$, такой что $\sigma_G=\id$, будем называть \emph{симметрией Дирихле} дроби $\cf(A)\in\gA_{n-1}$.
\end{definition}

\begin{definition}
  Оператор $G\in\textup{Sym}_{\Z}\big(\cf(A)\big)$, не являющийся симметрией Дирихле, будем называть \emph{палиндромической симметрией} дроби $\cf(A)$. Если множество палиндромических симметрией цепной дроби непусто, то такую цепную дробь будем называть \emph{палиндромичной}.
\end{definition}
 
\begin{definition}
  Симметрия $G\in\textup{Sym}_{\Z}\big(\cf(A)\big)$ называется \emph{циклической}, если $\sigma_G$ --- циклическая перестановка.
\end{definition}
 
 Очевидно, что все циклические симметрии цепной дроби $\cf(A)$ являются палиндромическими симметриями этой цепной дроби.

\begin{definition}
  Палиндромическая симметрия $G\in\textup{Sym}_{\Z}\big(\cf(A)\big)$ называется \emph{собственной}, если у оператора $G$ существует неподвижная точка на некотором парусе цепной дроби $\cf(A)$. Палиндромическая симметрия $G\in\textup{Sym}_{\Z}\big(\cf(A)\big)$, не являющаяся собственной, называется  \emph{несобственной}.
\end{definition}

Пусть $\partial(\cK(C))$ --- один из $2^n$ парусов цепной дроби $\cf(A)$. Благодаря теореме Дирихле об алгебраических единицах все симметрии Дирихле цепной дроби $\cf(A)$ образуют группу, и у этой группы есть подгруппа, относительно действия которой на парусе $\partial(\cK(C))$ возникает компактная фундаментальная область (см., например, \cite{german_tlyust_2}, \cite{korkina_2dim}). Таким образом можно говорить о \emph{периоде} паруса $\partial(\cK(C))$. В данной работе нас будут интересовать собственные циклические симметрии $\cf(A)$.

Для $n=2$, то есть для одномерных цепных дробей, палиндромичность напрямую связана с симметричностью периодов обыкновенных цепных дробей квадратичных иррациональностей. Критерий симметричности периода цепной дроби квадратичной иррациональности восходит к результатам Галуа \cite{galois}, Лежандра \cite{legendre}, Перрона \cite{perron_book} и Крайтчика \cite{kraitchik}. В работе \cite{german_tlyust} дано геометрическое доказательство этого критерия. При этом приходится рассматривать как собственные, так и несобственные симметрии. Однако, при $n=3$ любая палиндромичная цепная дробь обладает собственной циклической симметрией (см. \cite{german_tlyust_2}).

В упомянутых работах \cite{german_tlyust} и \cite{german_tlyust_2} доказываются следующие критерии наличия собственных циклических симметрий у алгебраических цепных дробей:

\begin{proposition}\label{two_dimension}
 Пусть $\cf(l_1,l_2)\in\gA_1$ и пусть подпространство $l_1$ порождено вектором $(1, \alpha)$. Тогда $\cf(l_1, l_2)$ имеет собственную циклическую симметрию в том и только в том случае, если существует такое алгебраическое число $\omega$ степени $2$ со своим сопряжённым $\omega'$, что выполнено хотя бы одно из следующих условий:

  \textup{(а)} $(1, \alpha) \sim(1,\omega):\hskip 6.5mm \trace(\omega)=\omega + \omega' =0$;

  \textup{(б)} $(1, \alpha) \sim(1,\omega):\hskip 6.5mm \trace(\omega)=\omega + \omega' =1$.
  
\end{proposition}

\begin{proposition}\label{three_dimension}
 Пусть $\cf(l_1,l_2,l_3)\in\gA_2$ и пусть подпространство $l_1$ порождено вектором $(1, \alpha, \beta)$. Тогда $\cf(l_1, l_2, l_3)$ имеет собственную циклическую симметрию в том и только в том случае, если существует такое алгебраическое число $\omega$ степени $3$ со своими сопряжёнными $\omega'$ и $\omega''$, что выполнено хотя бы одно из следующих условий:

  \textup{(а)} $(1, \alpha, \beta)\sim(1, \omega, \omega'):\hskip 6.5mm \trace(\omega)=\omega + \omega' + \omega'' =0$;

  \textup{(б)} $(1, \alpha, \beta)\sim(1, \omega, \omega'):\hskip 6.5mm \trace(\omega)=\omega + \omega' + \omega'' =1$.

\noindent При выполнении утверждения \textup{(а)} или \textup{(б)} кубическое расширение $\Q(\alpha, \beta)$ будет нормальным.
\end{proposition}

В этих формулировках $\vec v_1\sim\vec v_2$ для векторов из $\R^n$ означает существование такого оператора $X\in\Gl_n(\Z)$ и такого ненулевого $\mu\in\R$, что $X\vec v_1=\mu\vec v_2$.

\section{Формулировки основных результатов}\label{results}

Первым результатом данной работы является доказательство существования палиндромичных цепных дробей в произвольной размерности.

\begin{theorem}\label{main_t_n}
  Для любого целого $n > 1$ существует $(n-1)$-мерная цепная дробь $\cf(A)$, обладающая собственной циклической палиндромической симметрией.
\end{theorem}

 Второй результат данной работы обобщает предложения \ref{two_dimension} и \ref{three_dimension} на случай $n=4$:
\begin{theorem}\label{main_t_4}
  Пусть $\cf(l_1,l_2,l_3,l_4)\in\gA_3$ и пусть подпространство $l_1$ порождено вектором $(1, \alpha, \beta, \gamma)$. Тогда $\cf(l_1, l_2, l_3, l_4)$  имеет собственную циклическую симметрию в том и только в том случае, если существует такое алгебраическое число $\omega$ степени $4$ со своими сопряжёнными $\omega'$, $\omega''$, $\omega'''$ что выполнено хотя бы одно из следующих условий:

  \textup{(1)} $(1, \alpha, \beta, \gamma)\sim(1, \omega, \omega', \omega''):\hskip 14.5mm \trace(\omega)=\omega + \omega' + \omega'' + \omega'''=0$;
  
  \textup{(2)} $(1, \alpha, \beta, \gamma)\sim(1, \omega, \omega', \omega''):\hskip 14.5mm \trace(\omega)=\omega + \omega' + \omega'' + \omega'''=1$;
   
  \textup{(3)} $(1, \alpha, \beta, \gamma)\sim(1, \omega, \omega', \omega''):\hskip 14.5mm \trace(\omega)=\omega + \omega' + \omega'' + \omega'''=2$;

  \textup{(4)} $(1, \alpha, \beta, \gamma)\sim(1, \omega, \omega', \frac{\omega + \omega''}{2}):\hskip 10.3mm \trace(\omega)=\omega + \omega' + \omega'' + \omega'''=0$;
  
  \textup{(5)} $(1, \alpha, \beta, \gamma)\sim(1, \omega, \omega', \frac{\omega + \omega''}{2}):\hskip 10.3mm \trace(\omega)=\omega + \omega' + \omega'' + \omega'''=2$;
   
  \textup{(6)} $(1, \alpha, \beta, \gamma)\sim(1, \omega, \omega', \frac{\omega + \omega'' + 1}{2}):\hskip 6.5mm \trace(\omega)=\omega + \omega' + \omega'' + \omega'''=0$;

  \textup{(7)} $(1, \alpha, \beta, \gamma)\sim(1, \omega, \omega', \frac{\omega + \omega'' + 1}{2}):\hskip 6.5mm \trace(\omega)=\omega + \omega' + \omega'' + \omega'''=2$.
  
  \noindent При выполнении какого-то из утверждений \textup{(1)} -- \textup{(7)} расширение $\Q(\alpha, \beta, \gamma)$ степени $4$ будет нормальным.

\end{theorem}

Теорему \ref{main_t_n} мы докажем в параграфе \ref{sec_3}, а в параграфе \ref{sec_4} мы докажем теорему  \ref{main_t_4}. 

\begin{remark}
  В размерностях $n=2, 3$ все палиндромические симметрии цепной дроби $\cf(A)$  являются циклическими симметриями этой цепной дроби, поэтому в формулировках предложений  \ref{two_dimension} и \ref{three_dimension} термин "циклическая" можно заменить на термин "палиндромическая". В размерности $n=4$ это уже не так. Однако, полная классификация палиндромических симметрий слишком громоздка для данной статьи. Этому вопросу будет посвящено отдельное исследование.
\end{remark}

\section{Существование палиндромических симметрий для конечных вполне вещественных циклических расширений Галуа}\label{sec_3}

Здесь и далее будем обозначать через $\norm(\alpha)$ и $\trace(\alpha)$ соответственно норму $\norm_{\Q(\alpha)/\Q}(\alpha)$ и след $\trace_{\Q(\alpha)/\Q}(\alpha)$ алгебраического числа $\alpha$. Если задана дробь $\cf(l_1, \ldots, l_n)=\cf(A)\in\gA_{n-1}$, будем считать, что подпространство $l_1$ порождается вектором  $\vec l_1=(1,\alpha_1, \dots, \alpha_{n-1})$ (данное допущение корректно в силу предложения \ref{prop:more_than_pelle_n_dim}). Тогда из предложения \ref{prop:more_than_pelle_n_dim} следует, что числа $1,\alpha_1, \dots, \alpha_{n-1}$ образуют базис поля $K=\Q(\alpha_1, \dots, \alpha_{n-1})$ над $\Q$ и каждое $l_i$ порождается вектором $\vec l_i=(1,\sigma_i(\alpha_1), \dots, \sigma_i(\alpha_{n-1}))$, где $\sigma_1(=\id),\sigma_2, \dots, \sigma_n$ --- все вложения $K$ в $\R$.

Рассмотрим матрицу вида
  \begin{equation}\label{matrix_ord_n_0}
  \begin{pmatrix}
     0 & \mu_{1} & 0 & \dots & 0 & 0\\
    0 & 0 & \mu_{2} & \dots & 0 & 0 \\
    0 & 0 & 0 & \dots & 0 & 0 \\
    \dots &  \dots & \dots & \dots &\dots &\dots \\
    0 & 0 & 0 & \dots & 0 & \mu_{n-1}\\
    \mu_{n} & 0 & 0 & \dots & 0 & 0\\
   \end{pmatrix}.
\end{equation}
\begin{lemma}\label{about_properity}
 Пусть $G$ --- циклическая симметрия $\cf(l_1, \ldots, l_n) \in\gA_{n-1}$ и матрица оператора $G$ в базисе $\vec{l}_{1},  \ldots, \vec{l}_{n}$ имеет вид \ref{matrix_ord_n_0}. Тогда $G$ является собственной циклической симметрией дроби $\cf(l_1, \ldots, l_n)$ в том и только том случае, если $\mu_{1}\mu_{2}\ldots\mu_{n} = 1$.
 \end{lemma}
 
 \begin{proof}
Пусть $G$ является собственной циклической симметрией дроби $\cf(l_1, \ldots, l_n)$. Тогда существуют такие числа $\varepsilon_{1}, \varepsilon_{2}, \ldots, \varepsilon_{n}$ из множества $\{-1, 1\}$, что
\[G(\varepsilon_{1}\vec{l}_1, \varepsilon_{2}\vec{l}_2, \ldots, \varepsilon_{n}\vec{l}_{n}) =  \big(\mu_{n}\varepsilon_{1}\vec l_{n},\mu_1\varepsilon_{2}\vec l_1, \mu_2\varepsilon_{3}\vec l_2, \ldots, \mu_{n-1}\varepsilon_{n}\vec l_{n-1}\big),\]
и выполняются неравенства
\[\mu_1\frac{\varepsilon_{2}}{\varepsilon_{1}} > 0, \, \, \mu_2\frac{\varepsilon_{3}}{\varepsilon_{2}} > 0, \, \, \ldots, \, \, \mu_{n-1}\frac{\varepsilon_{n}}{\varepsilon_{n-1}} > 0, \, \, \mu_{n}\frac{\varepsilon_{1}}{\varepsilon_{n}} > 0.\]
Стало быть, $\mu_{1}\mu_{2}\ldots\mu_{n} > 0$, а значит $\mu_{1}\mu_{2}\ldots\mu_{n}  = 1$.

Если $\mu_{1}\mu_{2}\ldots\mu_{n} = 1$, то оператор $G$ имеет собственное направление, которое соответствует собственному значению $1$ и лежит внутри некоторого конуса $C \in \mathcal{C}(l_1, \ldots, l_n)$. 
 \end{proof}
 
Мы будем обозначать через $\gA_{n-1}'$ множество всех $(n-1)$-мерных алгебраических цепных дробей, для которых поле $K$ из предложения \ref{prop:more_than_pelle_n_dim} --- вполне вещественное циклическое расширение Галуа. Пусть $\sigma$ --- образующая группы Галуа $\gal(K/\Q)$. Также мы выбираем такую нумерацию прямых $l_1, \ldots, l_n$, что если через $\big(\vec l_1,\vec l_2, \dots, \vec l_{n-1}, \vec l_{n} \big)$ обозначить матрицу со столбцами $\vec l_1,\vec l_2, \dots, \vec l_{n-1}, \vec l_{n}$, то получим
 \[
  \big(\vec l_1,\vec l_2, \dots, \vec l_{n-1}, \vec l_{n} \big)=
  \begin{pmatrix}
    1 & 1 & \dots & 1 & 1 \\
    \alpha_1 & \sigma(\alpha_1) & \dots & \sigma^{n-2}(\alpha_1) & \sigma^{n-1}(\alpha_1) \\
    \alpha_2 & \sigma(\alpha_2) & \dots & \sigma^{n-2}(\alpha_2) & \sigma^{n-1}(\alpha_2) \\
    \dots & \dots & \dots & \dots & \dots \\
    \alpha_{n-1} & \sigma(\alpha_{n-1}) & \dots & \sigma^{n-2}(\alpha_{n-1}) & \sigma^{n-1}(\alpha_{n-1})
  \end{pmatrix}.
\]
Определим следующий класс $(n-1)$-мерных алгебраических цепных дробей:
\begin{align*}
  \mathbf{CF} & = \Big\{ \cf(l_1, \ldots, l_n)\in\gA_{n-1}' \,\Big|\, \alpha_{j} = \prod\limits_{k = 0}^{j-1}\sigma^{k}(\alpha_1), \, \ \norm(\alpha_1)=1\}.
\end{align*}
Положим
\[
  H =
  \begin{pmatrix}
    0 & 1 & 0 & \dots & 0 & 0\\
    0 & 0 & 1 & \dots & 0 & 0 \\
    0 & 0 & 0 & \dots & 0 & 0 \\
    \dots &  \dots & \dots & \dots &\dots &\dots \\
    0 & 0 & 0 & \dots & 0 & 1\\
    1 & 0 & 0 & \dots & 0 & 0\\
  \end{pmatrix}.\]
\begin{lemma}\label{oper_eq_n}
  Пусть $\cf(l_1, \ldots, l_n)\in\gA_{n-1}$. Тогда следующие два утверждения эквивалентны:
  
 \textup{(a)} $\cf(l_{1}, \ldots, l_{n})$ принадлежит классу $\mathbf{CF}$;
 
 \textup{(б)} $H$ --- собственная палиндромическая симметрия $\cf(l_1, \ldots, l_n)$ и 
 \[\sigma_{H} = (1, 2, \ldots, n-1, n).\]
\end{lemma}
\begin{proof}
  В силу леммы \ref{about_properity} оператор $H\in\Gl_n(\Z)$ является собственной палиндромической симметрией $\cf(A)$ и $\sigma_{H} = (1, 2, \ldots, n-1, n)$ тогда и только тогда, когда существуют такие действительные  числа $\mu_1,\mu_2, \ldots,\mu_n$, что $\mu_1\mu_2 \ldots \mu_n = 1$  и
 \begin{equation}\label{indexes}
 H\big(\vec l_{1},\vec l_{2}, \ldots, \vec l_{n-1}, \vec l_{n}\big)=\big(\mu_2\vec l_{2},\mu_3\vec l_{3},\ldots,\mu_n\vec l_{n},\mu_1\vec l_{1}\big).
 \end{equation}
 
   Пусть $\cf(l_{1}, \ldots, l_{n}) \in \mathbf{CF}$. Тогда
 \[  \big(\vec l_{1},\vec l_{2}, \ldots, \vec l_{n-1}, \vec l_{n}\big) =  \]
 \[ =
   \left( \begin{smallmatrix}
      1\phantom{-}\phantom{-} & 1\phantom{-}\phantom{-} & \ldots\phantom{-}\phantom{-} & 1\phantom{-}\phantom{-} & 1 \\
      \alpha_1\phantom{-}\phantom{-} & \sigma(\alpha_1)\phantom{-}\phantom{-} & \ldots\phantom{-}\phantom{-}& \sigma^{n-2}(\alpha_1)\phantom{-}\phantom{-} & \sigma^{n-1}(\alpha_1) \\
      \alpha_1\sigma(\alpha_1)\phantom{-}\phantom{-} & \sigma(\alpha_1)\sigma^{2}(\alpha_1)\phantom{-}\phantom{-} & \ldots\phantom{-}\phantom{-}& \sigma^{n-2}(\alpha_1)\sigma^{n-1}(\alpha_1)\phantom{-}\phantom{-} & \sigma^{n-1}(\alpha_1)\alpha_1 \\
        \ldots\phantom{-}\phantom{-} &  \ldots\phantom{-}\phantom{-} & \ldots\phantom{-}\phantom{-} & \ldots\phantom{-}\phantom{-} & \ldots \\
       \prod\limits_{k = 1}^{j-1}\sigma^{k-1}(\alpha_1) & \prod\limits_{k = 2}^{(j-1) + 1}\sigma^{k-1}(\alpha_1) & \ldots\phantom{-}\phantom{-} & \prod\limits_{k = n-1}^{(j-1) + n-2}\sigma^{k-1}(\alpha_1) & \prod\limits_{k = n}^{(j-1) + n - 1}\sigma^{k-1}(\alpha_1) \\ 
        \ldots\phantom{-}\phantom{-} &  \ldots\phantom{-}\phantom{-} & \ldots\phantom{-}\phantom{-} & \ldots\phantom{-}\phantom{-} & \ldots \\
    \end{smallmatrix}\right). \]
То есть, 
 \[  \big(\vec l_{1},\vec l_{2}, \ldots, \vec l_{n-1}, \vec l_{n}\big) = (a_{ji}), \] 
 где $a_{ji} = \prod\limits_{k = i}^{(j-1) + i - 1}\sigma^{k-1}(\alpha_1)$ при $j = 2, \ldots, n$ и $a_{1i} = 1$ для любого $i=1, \ldots, n$. Стало быть,
\[   H\big(\vec l_{1},\vec l_{2}, \ldots, \vec l_{n-1}, \vec l_{n}\big) = \big(\alpha_1\vec{l}_2, \sigma(\alpha_1)\vec{l}_3, \ldots, \sigma^{n-2}(\alpha_1)\vec{l}_{n},  \sigma^{n-1}(\alpha_1)\vec{l}_1\big). \]
Следовательно, $H$ --- собственная палиндромическая симметрия $\cf(l_1, \ldots, l_n)$ и $\sigma_{H} = (1, 2, \ldots, n-1, n)$. 

Обратно, предположим, $H$ ---  собственная палиндромическая симметрия цепной дроби $\cf(l_1, \ldots, l_n)$ и $\sigma_{H} = (1, 2, \ldots, n-1, n)$. Тогда, поскольку выполняется соотношение \ref{indexes}, имеем
\[
    H\vec{l}_1=
    \begin{pmatrix}
      \alpha_1  \\
      \alpha_2 \\
      \ldots \\
      \alpha_{n-1} \\
      1
    \end{pmatrix}
    =\mu_2
    \begin{pmatrix}
      1  \\
      \sigma_{2}(\alpha_1) \\
      \ldots \\
      \sigma_{2}(\alpha_{n-2}) \\
     \sigma_{2}(\alpha_{n-1})
    \end{pmatrix}, 
  \]
  откуда $\mu_2 = \alpha_1$, $\alpha_2 = \alpha_1\sigma_{2}(\alpha_1)$, $\alpha_3 = \alpha_1\sigma_{2}(\alpha_2) = \alpha_1\sigma_{2}(\alpha_1)\sigma^{2}_{2}(\alpha_1)$, $\ldots$, $\alpha_{n-1} = \alpha_1\sigma_{2}(\alpha_1)\cdots\sigma^{n-2}_{2}(\alpha_1)$, $1 = \alpha_1\sigma_{2}(\alpha_1)\cdots\sigma^{n-1}_{2}(\alpha_1)$. Для любого $i=2, \ldots, n-1$, в силу соотношения \ref{indexes}, имеем
   \[
    H\vec{l}_i=
    \begin{pmatrix}
      \sigma_{i}(\alpha_1)  \\
      \sigma_{i}(\alpha_2) \\
      \ldots \\
       \sigma_{i}(\alpha_{n-1})  \\
      1
    \end{pmatrix}
    =\mu_{i+1}
    \begin{pmatrix}
      1  \\
      \sigma_{i+1}(\alpha_1) \\
      \ldots \\
      \sigma_{i+1}(\alpha_{n-2}) \\
     \sigma_{i+1}(\alpha_{n-1})
    \end{pmatrix}, 
  \]
   откуда $\mu_{i+1} = \sigma_{i}(\alpha_1)$ и 
  \[\sigma_{i+1}(\alpha_1) = \frac{\sigma_{i}(\alpha_2)}{\sigma_{i}(\alpha_1)}=\sigma_{i}(\sigma_{2}(\alpha_1)),\] 
  \[\sigma_{i+1}(\alpha_2) = \frac{\sigma_{i}(\alpha_3)}{\sigma_{i}(\alpha_1)}=\sigma_{i}(\sigma_{2}(\alpha_2)),\] 
  \[\ldots,\]
   \[\sigma_{i+1}(\alpha_{n-1}) = \frac{1}{\sigma_{i}(\alpha_1)}=\sigma_{i}(\sigma_{2}(\alpha_{n-1})).\] 
  Применяя индукцию, получаем, что $\sigma_{i+1}(\alpha_1) =  \sigma_{2}^{i}(\alpha_1)$, $\sigma_{i+1}(\alpha_2) =  \sigma_{2}^{i}(\alpha_2)$, $\ldots,$ $\sigma_{i+1}(\alpha_{n-1}) =  \sigma_{2}^{i}(\alpha_{n-1})$. Тогда $\cf(l_1, \ldots, l_n) \in \mathbf{CF}$, так как числа $1,\alpha_1, \dots, \alpha_{n-1}$ образуют базис поля $K$.
\end{proof}

\begin{proof}[Доказательство теоремы \ref{main_t_n}]
Для начала докажем существование конечного вполне вещественного циклического расширения Галуа степени $n$ поля $\Q$. В силу теоремы Дирихле об арифметической прогрессии существует такое  простое $p$, что $p\equiv 1{\pmod {2n}}$. Пусть $\zeta_{p}$ --- корень степени $p$ из единицы, и $E = \Q(\zeta_{p})$. Заметим, что $\gal(E/\Q) = \Z/(p-1)\Z$. Рассмотрим поле $K_0 = \Q(\zeta_{p} + \zeta^{-1}_{p})$. Тогда $K_0$ --- конечное вполне вещественное расширение поля $\Q$ степени $\frac{p-1}{2}$ (см. \cite{lehmer}) и $[E:K_0] = 2$, поскольку $x^2 - (\zeta_{p} + \zeta^{-1}_{p})x +1$ --- минимальный многочлен для $\zeta_{p}$ над $K_0$. Поскольку группа $\gal(E/\Q)$ является циклической, то все ее подгруппы нормальны, более того, все факторгруппы $\gal(E/\Q)$ по подгруппам $\gal(E/\Q)$ циклические, а значит  $K_0$ --- циклическое расширение Галуа поля $\Q$ в силу основной теоремы теории Галуа. Поскольку $n$ делит $\frac{p-1}{2}$, то циклическая группа $\gal(K_0/\Q)$ содержит подгруппу $F$ индекса $n$. Пусть $K = K_{0}^{F}$. Вновь применяя основную теорему теории Галуа получаем, что $K$ --- циклическое расширение Галуа поля $\Q$,  $[K:\Q] = [\gal(K_{0}/\Q):F] = n$. При этом поле $K \subset K_{0}$ вполне вещественное расширение поля $\Q$.

Пусть $K$ --- конечное вполне вещественное циклическое расширение Галуа степени $n$ поля $\Q$, а $\sigma$ --- порождающий элемент группы Галуа этого расширения. По теореме о нормальном базисе существует набор чисел 
\[\omega, \sigma(\omega), \dots, \sigma^{n-1}(\omega),\]
 являющийся базисом расширения $K$. Тогда набор чисел $1, \frac{\sigma(\omega)}{\omega}, \dots, \frac{\sigma^{n-1}(\omega)}{\omega}$ также образуют базис расширения $K$. Стало быть, в силу предложения \ref{prop:more_than_pelle_n_dim} вектор $(1, \frac{\sigma(\omega)}{\omega}, \dots, \frac{\sigma^{n-1}(\omega)}{\omega})$ является собственным для некоторого гиперболического оператора $A\in\Sl_n(\Z)$. Заметим, что для любого $j=2, \ldots, n-1$
\[\frac{\sigma(\omega)}{\omega}\sigma\Big(\frac{\sigma(\omega)}{\omega}\Big)\sigma^{2}\Big(\frac{\sigma(\omega)}{\omega}\Big)\cdots\sigma^{j-1}\Big(\frac{\sigma(\omega)}{\omega}\Big) = \frac{\sigma^{j}(\omega)}{\omega},\]
при этом $\norm(\frac{\sigma(\omega)}{\omega})=1$. Таким образом $\cf(A) \in \mathbf{CF}$. Для завершения доказательства достаточно применить лемму \ref{oper_eq_n}.
\end{proof}

\section{Палиндромические симметрии в случае $n=4$}\label{sec_4}

Отныне будем считать, что $n=4$, то есть будем рассматривать трехмерные цепные дроби. Напомним, что множество всех трехмерных алгебраических цепных дробей мы обозначаем через $\gA_3$. Далее, как и в параграфе \ref{sec_3}, будем считать, что если задана дробь $\cf(l_1,l_2,l_3, l_4)\in\gA_3$, то подпространство $l_i$ порождается вектором  $\vec l_i$, первая координата которого равна $1$, где $i=1,2,3,4$.

Пусть $G$ --- циклическая симметрия $\cf(l_1,l_2,l_3, l_4)\in\gA_3$.  Изменив при необходимости нумерацию подпространств $l_1, l_2, l_3, l_4$  можно рассмотреть такие вещественные числа $\mu_{1}$, $\mu_{2}$, $\mu_{3}$, $\mu_{4}$, что матрица оператора $G$ в базисе $\vec{l}_{1}, \vec{l}_{2}, \vec{l}_{3}, \vec{l}_{4}$ имеет вид
  \begin{equation}\label{matrix_ord_4_0}
   \begin{pmatrix}
     0 & 0 & 0 &  \mu_{1}\\
     \mu_{2} & 0 & 0 & 0\\
      0 & \mu_{3} & 0 & 0\\
       0 & 0 &  \mu_{4} & 0
   \end{pmatrix}.
 \end{equation}

 Лемма \ref{about_properity} в случае  $n=4$ при изменении нумерации подпространств $l_1, l_2, l_3, l_4$ приобретает следующий вид 
 \begin{corollary}\label{about_properity_n_4}
  Пусть $G$ --- циклическая симметрия $\cf(l_1,l_2,l_3, l_4)\in\gA_3$ и матрица оператора $G$ в базисе $\vec{l}_{1}, \vec{l}_{2}, \vec{l}_{3}, \vec{l}_{4}$ имеет вид \ref{matrix_ord_4_0}. Тогда $G$ является собственной циклической симметрией дроби $\cf(l_1,l_2, l_3, l_4)$ в том и только том случае, если $\mu_{1}\mu_{2}\mu_{3}\mu_{4} = 1$.
\end{corollary}

\begin{lemma}\label{rational_subspace}
  Пусть $G$ --- собственная циклическая симметрия $\cf(l_1,l_2,l_3, l_4)\in\gA_3$. Тогда собственные значения оператора $G$ равны $1$, $-1$, $i$ и $-i$. Более того, собственные подпространства $l_{+}$ (соответствующее собственному значению $1$), $l_{-}$ (соответствующее собственному значению $-1$) и $L$ (соответствующее собственным значениям $i$ и $-i$) являются рациональными. В частности, подпространство $L$ не содержит собственных для $G$ одномерных подпространств и для любого $\vec v \in L$ верно, что $G^2({\vec v}) = -\vec v$.
\end{lemma}
 
\begin{proof}
   Изменив при необходимости нумерацию подпространств $l_1, l_2, l_3, l_4$, можно считать, что в силу следствия \ref{about_properity_n_4} существуют такие вещественные числа $\mu_{1}$, $\mu_{2}$, $\mu_{3}$, $\mu_{4}$, что $\mu_{1}\mu_{2}\mu_{3}\mu_{4} = 1$ и матрица оператора $G$ в базисе $\vec{l}_{1}, \vec{l}_{2}, \vec{l}_{3}, \vec{l}_{4}$ имеет вид \ref{matrix_ord_4_0}. Так как $\chi_{G}(x) = x^4 - \mu_{1}\mu_{2}\mu_{3}\mu_{4}$, то собственные значения оператора $G$ равны $1$, $-1$, $i$ и $-i$, а значит у $G$ есть ровно два одномерных собственных подпространства и двумерное инвариантное подпространство, которое не содержит собственных для $G$ одномерных подпространств. Обозначим через $l_{+}$ рациональное одномерное собственное подпространство оператора $G$, соответствующее собственному значению $1$, через $l_{-}$ --- рациональное одномерное собственное подпространство оператора $G$, соответствующее собственному значению $-1$, а через $L$ --- двумерное инвариантное подпространство, соответствующее собственным значениям $i$ и $-i$. Покажем рациональность подпространства $L$.

Поскольку $l_{-} + l_{+} + L = \R^{4}$, то для любого вектора $\vec v \in \R^{4}$ существуют такие единственные вектора $\vec{p}(\vec v, l_{-}) \in  l_{-}$, $\vec{p}(\vec v,  l_{+})  \in  l_{+}$ и $\vec{p}(\vec v, L) \in L$, что выполняется равенство
\[\vec v = \vec{p}(\vec v, l_{-}) + \vec{p}(\vec v,  l_{+}) + \vec{p}(\vec v, L).\]
Заметим, что $\vec{p}\big(G^{2}(\vec{v}),  L\big) = \vec{p}(-\vec v, L)$, $\vec{p}\big(G^{2}(\vec{v}), l_{-}\big) = \vec{p}(\vec v, l_{-})$ и $\vec{p}\big(G^{2}(\vec{v}), l_{+}\big) = \vec{p}(\vec v, l_{+})$ для любого вектора $\vec v \in \R^{4}$. Таким образом, для любой точки $\vec{z} \in \Z^{4} \setminus (l_{+} \, + \, l_{-})$ ненулевые целочисленные вектора $\vec{z}-G^{2}(\vec{z})$ и $G(\vec{z})-G^{3}(\vec{z})$ лежат в двумерном подпространстве $L$. Эти два целочисленных вектора неколлинеарны, поскольку $G(\vec{z})-G^{3}(\vec{z}) =  G\big(\vec{z}-G^{2}(\vec{z})\big)$ и подпространство $L$ не содержит собственных для оператора $G$ одномерных подпространств.  Итак, мы показали, что подпространство $L$ рационально. 
\end{proof}

\begin{lemma}\label{main_lem}
  Пусть $G$ --- собственная циклическая симметрия $\cf(l_1,l_2,l_3, l_4)\in\gA_3$. Тогда существуют $\vec{z}_1$, $\vec{z}_2$, $\vec{z}_3$, $\vec{z}_4$ $\in$ $\Z^4$, такие что
\[G(\vec{z}_{1}) = \vec{z}_{2}, \, G(\vec{z}_{2}) = \vec{z}_{3}, \, G(\vec{z}_{3}) = \vec{z}_{4}, \, G(\vec{z}_{4}) = \vec{z}_{1}\] 
и выполняется хотя бы одно из следующих семи утверждений:

\textup{(1)} вектора $\vec{z}_{1}$, $\vec{z}_{2}$, $\vec{z}_{3}$, $\frac{1}{4}(\vec{z}_{1}+\vec{z}_{2}+\vec{z}_{3}+\vec{z}_{4})$, образуют базис решетки $\Z^4$;

\textup{(2)} вектора $\vec{z}_{1}$, $\vec{z}_{2}$,  $\vec{z}_{3}$, $\vec{z}_{4}$ образуют базис решетки $\Z^4$;

\textup{(3)} вектора $\vec{z}_{1}$, $\frac{1}{2}(\vec{z}_{1}+\vec{z}_{2})$, $\frac{1}{2}(\vec{z}_{1}+\vec{z}_{3})$, $\frac{1}{2}(\vec{z}_{1}+\vec{z}_{4})$ образуют базис решетки $\Z^4$;

\textup{(4)} вектора $\vec{z}_{1}$, $\vec{z}_{2}$, $\frac{1}{2}(\vec{z}_{1}+\vec{z}_{3})$, $\frac{1}{4}(\vec{z}_{1}+\vec{z}_{2}+\vec{z}_{3}+\vec{z}_{4})$ образуют базис решетки $\Z^4$;

\textup{(5)} вектора $\vec{z}_{1}$, $\vec{z}_{2}$,  $\frac{1}{2}(\vec{z}_{1}+\vec{z}_{3})$, $\frac{1}{2}(\vec{z}_{2}+\vec{z}_{4})$ образуют базис решетки $\Z^4$;

\textup{(6)} вектора $\vec{z}_{1}$, $\vec{z}_{2}$,  $\vec{z}_{3}$, $\frac{1}{2}(\vec{z}_{1} + \vec{z}_{3} + \vec{z}_{4} - \vec{z}_{2})$ образуют базис решетки $\Z^4$;

\textup{(7)} вектора $\vec{z}_{1}$, $\vec{z}_{2}$,  $\vec{z}_{3}$, $\frac{1}{2}(\vec{z}_{1}+\vec{z}_{2}) + \frac{1}{4}(\vec{z}_{1}+\vec{z}_{4} - \vec{z}_{3} - \vec{z}_{2})$ образуют базис решетки $\Z^4$.
\end{lemma}

\begin{proof}
Будем называть плоскость \emph{рациональной}, если множество содержащихся в нем целых точек является (аффинной) решеткой ранга, равного размерности этой плоскости.

Рассмотрим подпространства $l_{+}$, $l_{-}$ и $L$ из леммы \ref{rational_subspace} и положим $S = l_{-} + L$. Обозначим через $S_1$ ближайшую к $S$ рациональную гиперплоскость, параллельную $S$ и не совпадающую с $S$ (любую из двух). Тогда $G(S_1) = S_1$. Также обозначим через $\vec{p}$ точку пересечения гиперплоскости  $S_1$ и $l_{+}$, а через $l$ и $\pi$ прямую и плоскость, проходящие через точку $\vec{p}$ и параллельные $l_{-}$ и $L$ соответственно. Тогда, в силу леммы \ref{rational_subspace},  $G(\vec{p}) = \vec{p}$, $G(\vec{v} - \vec{p}) = \vec{p} - \vec{v}$  для любого вектора $\vec v \in l$ и  $G^{2}(\vec{v} - \vec{p}) = \vec{p} - \vec{v}$  для любого вектора $\vec v \in \pi$. 

Поскольку подпространство $L$ не содержит собственных для $G$ одномерных подпространств, то для произвольной точки $\vec{v} \in \pi \setminus l$ четырехугольник 
\[\textup{conv}\big(\vec{v}, G(\vec{v}), G^{2}(\vec{v}), G^{3}(\vec{v})\big)\]
 является параллелограммом, диагонали которого пересекаются в точке $\vec{p} = \frac{1}{2}\big(\vec{v} + G^{2}(\vec{v})\big) = \frac{1}{2}\big(G(\vec{v}) + G^{3}(\vec{v})\big)$.

Обозначим через $Q$ рациональную плоскость ближайшую к $\pi$, лежащую в гиперплоскости $S_1$, параллельную $\pi$ и не совпадающую с $\pi$. Поскольку $G(\vec{v} - \vec{p}) = \vec{p} - \vec{v}$  для любого вектора $\vec v \in l$, то $R = G(Q)$ и $Q$ --- рациональные плоскости ближайшие к $\pi$ и равноудаленные от нее, лежащие в гиперплоскости $S_1$ по разные стороны от $\pi$, параллельные $\pi$ и не совпадающие с $\pi$. Положим $\vec{p}^{Q} = Q \cap l$ и $\vec{p}^{R} = R \cap l$. Построим точки $\vec{z}_{1}$, $\vec{z}_{2}$, $\vec{z}_{3}$, $\vec{z}_{4}$ при помощи следующей итерационной процедуры. Возьмем произвольную целочисленную точку $\vec{v}_{1,1} \in Q \setminus l$. Введем обозначения $\vec{v}_{1,2} = G(\vec{v}_{1,1})$, $\vec{v}_{1,3} = G^{2}(\vec{v}_{1,1})$, $\vec{v}_{1,4} = G^{3}(\vec{v}_{1,1})$. Пусть точки  $\vec{v}^{\pi}_{1,1}$, $\vec{v}^{\pi}_{1,2}$, $\vec{v}^{\pi}_{1,3}$, $\vec{v}^{\pi}_{1,4}$ --- проекции параллельные $l$ на плоскость $\pi$ точек $\vec{v}_{1,1}$, $\vec{v}_{1,2}$, $\vec{v}_{1,3}$, $\vec{v}_{1,4}$ соответственно. Также обозначим через $\vec{v}^{R}_{1,1}$ и $\vec{v}^{R}_{1,3}$ --- проекции параллельные $l$ на плоскость $R$ точек $\vec{v}_{1,1}$ и $\vec{v}_{1,3}$, а через $\vec{v}^{Q}_{1,2}$ и $\vec{v}^{Q}_{1,4}$ --- проекции параллельные $l$ на плоскость $Q$ точек $\vec{v}_{1,2}$ и $\vec{v}_{1,4}$. По доказанному выше множество $\Delta^{\pi}_{1} = \textup{conv}(\vec{v}^{\pi}_{1,1}, \vec{v}^{\pi}_{1,2}, \vec{v}^{\pi}_{1,3}, \vec{v}^{\pi}_{1,4})$ является параллелограммом, диагонали которого пересекаются в точке $\vec{p} = \frac{1}{4}(\vec{v}_{1,1} + \vec{v}_{1,2} + \vec{v}_{1,3} + \vec{v}_{1,4})$. Таким образом множества $\Delta^{Q}_{1} = \textup{conv}(\vec{v}_{1,1}, \vec{v}^{Q}_{1,2}, \vec{v}_{1,3}, \vec{v}^{Q}_{1,4})$ и $\Delta^{R}_{1} = \textup{conv}(\vec{v}^{R}_{1,1}, \vec{v}_{1,2}, \vec{v}^{R}_{1,3}, \vec{v}_{1,4})$ также являются параллелограммами. Заметим, что $\vec{p}^{Q} = \frac{1}{2}(\vec{v}_{1,1} + \vec{v}_{1,3})$ и $\vec{p}^{R} = \frac{1}{2}(\vec{v}_{1,2} + \vec{v}_{1,4})$. 

Предположим, мы построили параллелограммы $\Delta^{\pi}_{j}$, $\Delta^{Q}_{j}$ и $\Delta^{R}_{j}$. Если на плоскостях $Q$ и $R$ существует целая точка, не совпадающая с точками $\vec{p}^{Q}$, $\vec{p}^{R}$ и ни с какой из вершин параллелограммов $\Delta^{Q}_{j}$ и $\Delta^{R}_{j}$, при этом лежащая в одном из этих параллелограммов (без ограничения общности, внутри $\Delta^{Q}_{j}$), то обозначим её через $\vec{v}_{j+1,1}$. Введем обозначения $\vec{v}_{j+1,2} = G(\vec{v}_{j+1,1})$, $\vec{v}_{j+1,3} = G^{2}(\vec{v}_{j+1,1})$, $\vec{v}_{j+1,4} = G^{3}(\vec{v}_{j+1,1})$. Пусть точки  $\vec{v}^{\pi}_{j+1,1}$, $\vec{v}^{\pi}_{j+1,2}$, $\vec{v}^{\pi}_{j+1,3}$, $\vec{v}^{\pi}_{j+1,4}$ --- проекции параллельные $l$ на плоскость $\pi$ точек $\vec{v}_{j+1,1}$, $\vec{v}_{j+1,2}$, $\vec{v}_{j+1,3}$, $\vec{v}_{j+1,4}$ соответственно. Также обозначим через $\vec{v}^{R}_{j+1,1}$ и $\vec{v}^{R}_{j+1,3}$ --- проекции параллельные $l$  на плоскость $R$ точек $\vec{v}_{j+1,1}$ и $\vec{v}_{j+1,3}$, а через $\vec{v}^{Q}_{j+1,2}$ и $\vec{v}^{Q}_{j+1,4}$ --- проекции параллельные $l$ на плоскость $Q$ точек $\vec{v}_{j+1,2}$ и $\vec{v}_{j+1,4}$. Определим параллелограммы 
\[\Delta^{\pi}_{j+1} = \textup{conv}(\vec{v}^{\pi}_{j+1,1}, \vec{v}^{\pi}_{j+1,2}, \vec{v}^{\pi}_{j+1,3}, \vec{v}^{\pi}_{j+1,4}),\]
\[\Delta^{Q}_{j+1} = \textup{conv}(\vec{v}_{j+1,1}, \vec{v}^{Q}_{j+1,2}, \vec{v}_{j+1,3}, \vec{v}^{Q}_{j+1,4}),\] 
\[\Delta^{R}_{j+1} = \textup{conv}(\vec{v}^{R}_{j+1,1}, \vec{v}_{j+1,2}, \vec{v}^{R}_{j+1,3}, \vec{v}_{j+1,4}).\] 
При этом $\vec{p} = \frac{1}{4}(\vec{v}_{j+1,1} + \vec{v}_{j+1,2} + \vec{v}_{j+1,3} + \vec{v}_{j+1,4})$, $\vec{p}^{Q} = \frac{1}{2}(\vec{v}_{j+1,1} + \vec{v}_{j+1,3})$ и $\vec{p}^{R} = \frac{1}{2}(\vec{v}_{j+1,2} + \vec{v}_{j+1,4})$. 

Последовательность троек $(\Delta^{\pi}_{j}, \Delta^{Q}_{j}, \Delta^{R}_{j})$ конечна. Пусть $(\Delta^{\pi}_{k}, \Delta^{Q}_{k}, \Delta^{R}_{k})$ --- последний её элемент. Положим $\vec{z}_{1} = \vec{v}_{k,1}$, $\vec{z}_{2} = \vec{v}_{k,2}$, $\vec{z}_{3} = \vec{v}_{k,3}$, $\vec{z}_{4} = \vec{v}_{k,4}$, $\vec{z}^{\pi}_{1} = \vec{v}^{\pi}_{k,1}$, $\vec{z}^{\pi}_{2} = \vec{v}^{\pi}_{k,2}$, $\vec{z}^{\pi}_{3} = \vec{v}^{\pi}_{k,3}$, $\vec{z}^{\pi}_{4} = \vec{v}^{\pi}_{k,4}$, $\vec{z}^{R}_{1} = \vec{v}^{R}_{k,1}$, $\vec{z}^{R}_{3} = \vec{v}^{R}_{k,3}$, $\vec{z}^{Q}_{2} = \vec{v}^{Q}_{k,2}$, $\vec{z}^{Q}_{4} = \vec{v}^{Q}_{k,4}$. Покажем, что множество $(\Delta^{\pi}_{k} \cup \Delta^{Q}_{k} \cup \Delta^{R}_{k}) \cap \Z^{4}$ совпадает с одним из множеств 
\[\{\vec{p}, \, \vec{z}_{1}, \, \vec{z}^{Q}_{2}, \, \vec{z}_{3}, \, \vec{z}^{Q}_{4}, \, \vec{z}^{R}_{1}, \, \vec{z}_{2}, \, \vec{z}^{R}_{3}, \, \vec{z}_{4} \},\]  
\[\{\vec{z}_{1}, \, \vec{z}_{3}, \, \vec{z}_{2}, \, \vec{z}_{4}\},\]
\[\{\vec{z}_{1}, \, \vec{z}^{Q}_{2}, \, \vec{z}_{3}, \, \vec{z}^{Q}_{4}, \, \vec{z}^{R}_{1}, \, \vec{z}_{2}, \, \vec{z}^{R}_{3}, \, \vec{z}_{4}, \, \vec{p}^{Q}, \, \vec{p}^{R}, \, \frac{\vec{z}^{\pi}_{1} + \vec{z}^{\pi}_{2}}{2}, \, \frac{\vec{z}^{\pi}_{2} + \vec{z}^{\pi}_{3}}{2}, \, \frac{\vec{z}^{\pi}_{3} + \vec{z}^{\pi}_{4}}{2}, \, \frac{\vec{z}^{\pi}_{4} + \vec{z}^{\pi}_{1}}{2}\},\]
\[\{\vec{p}, \, \vec{z}_{1}, \, \vec{z}^{Q}_{2}, \, \vec{z}_{3}, \, \vec{z}^{Q}_{4}, \, \vec{z}^{R}_{1}, \, \vec{z}_{2}, \, \vec{z}^{R}_{3}, \, \vec{z}_{4}, \, \vec{p}^{Q}, \, \vec{p}^{R}, \, \vec{z}^{\pi}_{1}, \, \vec{z}^{\pi}_{2}, \, \vec{z}^{\pi}_{3}, \, \vec{z}^{\pi}_{4}\},\]  
\[\{\vec{z}_{1}, \, \vec{z}^{Q}_{2}, \, \vec{z}_{3}, \, \vec{z}^{Q}_{4}, \, \vec{z}^{R}_{1}, \, \vec{z}_{2}, \, \vec{z}^{R}_{3}, \, \vec{z}_{4}, \, \vec{p}^{Q}, \, \vec{p}^{R}\},\]
\[\{\vec{z}_{1}, \, \vec{z}^{Q}_{2}, \, \vec{z}_{3}, \, \vec{z}^{Q}_{4}, \, \vec{z}^{R}_{1}, \, \vec{z}_{2}, \, \vec{z}^{R}_{3}, \, \vec{z}_{4}\},\]
\[\{\vec{z}_{1}, \, \vec{z}^{Q}_{2}, \, \vec{z}_{3}, \, \vec{z}^{Q}_{4}, \, \vec{z}^{R}_{1}, \, \vec{z}_{2}, \, \vec{z}^{R}_{3}, \, \vec{z}_{4}, \, \vec{z}^{\pi}_{1}, \, \vec{z}^{\pi}_{2}, \, \vec{z}^{\pi}_{3}, \, \vec{z}^{\pi}_{4}\}. \]

Для начала покажем, что 
\[(\Delta^{\pi}_{k}\cap \Z^{4}) \subset \{\vec{p}, \vec{z}^{\pi}_{1}, \, \vec{z}^{\pi}_{2}, \, \vec{z}^{\pi}_{3}, \, \vec{z}^{\pi}_{4}, \, \frac{\vec{z}^{\pi}_{1} + \vec{z}^{\pi}_{2}}{2}, \, \frac{\vec{z}^{\pi}_{2} + \vec{z}^{\pi}_{3}}{2}, \, \frac{\vec{z}^{\pi}_{3} + \vec{z}^{\pi}_{4}}{2}, \, \frac{\vec{z}^{\pi}_{4} + \vec{z}^{\pi}_{1}}{2}\}.\]
Предположим, что это не так. Без ограничения общности, будем считать, что существует целая точка $\vec{w}$ лежащая в параллелограмме $\textup{conv}(\vec{z}^{\pi}_{1}, \frac{\vec{z}^{\pi}_{1} + \vec{z}^{\pi}_{2}}{2}, \vec{p}, \frac{\vec{z}^{\pi}_{4} + \vec{z}^{\pi}_{1}}{2})$ и не совпадающая ни с какой вершиной этого параллелограмма. Если $\vec{w}$ является точкой пересечения диагоналей параллелограмма $\textup{conv}(\vec{z}^{\pi}_{1}, \frac{\vec{z}^{\pi}_{1} + \vec{z}^{\pi}_{2}}{2}, \vec{p}, \frac{\vec{z}^{\pi}_{4} + \vec{z}^{\pi}_{1}}{2})$, то целая точка  $\vec{z}_{1} + \big(G(\vec{w}) - \vec{w}\big)$, не совпадающая с точкой $\vec{p}^{Q}$ и ни с какой из вершин параллелограмма $\Delta^{Q}_{k}$, лежит в параллелограмме $\Delta^{Q}_{k}$. Если же $\vec{w}$ не является точкой пересечения диагоналей параллелограмма $\textup{conv}(\vec{z}^{\pi}_{1}, \frac{\vec{z}^{\pi}_{1} + \vec{z}^{\pi}_{2}}{2}, \vec{p}, \frac{\vec{z}^{\pi}_{4} + \vec{z}^{\pi}_{1}}{2})$, то целая точка  $\vec{z}_{1} + \big(G^{2}(\vec{w}) - \vec{w}\big)$, не совпадающая с точкой $\vec{p}^{Q}$ и ни с какой из вершин параллелограмма $\Delta^{Q}_{k}$, также лежит в параллелограмме $\Delta^{Q}_{k}$. В обоих случаях получаем противоречие построению тройки параллелограммов $(\Delta^{\pi}_{k}, \Delta^{Q}_{k}, \Delta^{R}_{k})$.

\textbf{А.} Предположим, что плоскость $\pi$ --- рациональная. Покажем, что этот случай разбивается на 4 принципиально разных случая.
 
\textbf{А.1.} Предположим, что $\vec{p} \in \Z^{4}$. В таком случае каждая из точек $\vec{z}^{R}_{1}$, $\vec{z}^{R}_{3}$, $\vec{z}^{Q}_{2}$, $\vec{z}^{Q}_{4}$ принадлежит решетке $\Z^{4}$. Ни одна из точек  $\frac{\vec{z}^{\pi}_{1} + \vec{z}^{\pi}_{2}}{2}$, $\frac{\vec{z}^{\pi}_{2} + \vec{z}^{\pi}_{3}}{2}$, $\frac{\vec{z}^{\pi}_{3} + \vec{z}^{\pi}_{4}}{2}$, $\frac{\vec{z}^{\pi}_{4} + \vec{z}^{\pi}_{1}}{2}$ не принадлежит решетке $\Z^{4}$, так как, в противном случае, середины сторон параллелограммов $\Delta^{Q}_{k}$ и $\Delta^{R}_{k}$ принадлежат решетке $\Z^{4}$. Теперь рассмотрим следующие случаи:

\textbf{А.1.1.} (будет соответствовать утверждению (1)) $\vec{p}^{Q} \notin \Z^{4}$, а значит и $\vec{p}^{R} \notin \Z^{4}$. Так как $\vec{p} \in \Z^{4}$, то ни какая из точек $\vec{z}^{\pi}_{1}$, $\vec{z}^{\pi}_{2}$, $\vec{z}^{\pi}_{3}$, $\vec{z}^{\pi}_{4}$ не принадлежит решетке $\Z^{4}$. Тогда
\[(\Delta^{\pi}_{k} \cup \Delta^{Q}_{k} \cup \Delta^{R}_{k}) \cap \Z^{4} = \{\vec{p}, \, \vec{z}_{1}, \, \vec{z}^{Q}_{2}, \, \vec{z}_{3}, \, \vec{z}^{Q}_{4}, \, \vec{z}^{R}_{1}, \, \vec{z}_{2}, \, \vec{z}^{R}_{3}, \, \vec{z}_{4} \},\]
и набор векторов $\vec{z}_{1}$, $\vec{z}_{2}$, $\vec{z}_{3}$, $\vec{p} = \frac{1}{4}(\vec{z}_{1}+\vec{z}_{2}+\vec{z}_{3}+\vec{z}_{4})$ образует базис решетки $\Z^{4}$, а значит выполняется утверждение (1).

 \begin{figure}[h]
  \centering
  \begin{tikzpicture}[x=10mm, y=7mm, z=-5mm,scale=0.7]
    \begin{scope}[rotate around x=0]
    
    \draw[->] [very thin] (-5,0,0) -- (5,0,0);
    \draw[->] [very thin] (0,-7,0) -- (0,7,0);
    \draw[->] [very thin] (0,0,-7) -- (0,0,7);

    \node[fill=black,circle,inner sep=1pt,opacity=0.5] at (0,0,0) {};
    \node[fill=black,circle,inner sep=1pt,opacity=0.5] at (0,0,2) {};
    \node[fill=black,circle,inner sep=1pt,opacity=0.5] at (0,0,-2) {};
    \node[fill=black,circle,inner sep=1pt,opacity=0.5] at (2,0,0) {};
    \node[fill=black,circle,inner sep=1pt,opacity=0.5] at (2,0,-2) {};
    \node[fill=black,circle,inner sep=1pt,opacity=0.5] at (-2,0,0) {};
    \node[fill=black,circle,inner sep=1pt,opacity=0.5] at (-2,0,2) {};
    \node[fill=black,circle,inner sep=1pt,opacity=0.5] at (0,-5,0) {};
    \node[fill=black,circle,inner sep=1pt,opacity=0.5] at (0,-5,2) {};
    \node[fill=black,circle,inner sep=1pt,opacity=0.5] at (0,-5,-2) {};
    \node[fill=black,circle,inner sep=1pt,opacity=0.5] at (2,-5,0) {};
    \node[fill=black,circle,inner sep=1pt,opacity=0.5] at (2,-5,2) {};
    \node[fill=black,circle,inner sep=1pt,opacity=0.5] at (2,-5,-2) {};
    \node[fill=black,circle,inner sep=1pt,opacity=0.5] at (-2,-5,0) {};
    \node[fill=black,circle,inner sep=1pt,opacity=0.5] at (-2,-5,2) {};
    \node[fill=black,circle,inner sep=1pt,opacity=0.5] at (-2,-5,-2) {};
    \node[fill=black,circle,inner sep=1pt,opacity=0.5] at (0,5,0) {};
    \node[fill=black,circle,inner sep=1pt,opacity=0.5] at (0,5,2) {};
    \node[fill=black,circle,inner sep=1pt,opacity=0.5] at (0,5,-2) {};
    \node[fill=black,circle,inner sep=1pt,opacity=0.5] at (2,5,0) {};
    \node[fill=black,circle,inner sep=1pt,opacity=0.5] at (2,5,2) {};
    \node[fill=black,circle,inner sep=1pt,opacity=0.5] at (2,5,-2) {};
    \node[fill=black,circle,inner sep=1pt,opacity=0.5] at (-2,5,0) {};
    \node[fill=black,circle,inner sep=1pt,opacity=0.5] at (-2,5,2) {};
    \node[fill=black,circle,inner sep=1pt,opacity=0.5] at (-2,5,-2) {};
       
    \coordinate (z_1) at (-2,0,0);
    \coordinate (z_3) at (0,-5,2);
    \coordinate (z_2) at (0,5,0);
    \coordinate (z_4) at (2,0,-2);
    
    \coordinate (p_pl_1) at ($ (z_1)!1/2!(z_3) $);
    \coordinate (p_mi_1) at ($ (z_2)!1/2!(z_4) $);
    \coordinate (z_1_z_2_half) at ($ (z_1)!1/2!(z_2) $);
    \coordinate (z_2_z_3_half) at ($ (z_2)!1/2!(z_3) $);
    \coordinate (z_3_z_4_half) at ($ (z_3)!1/2!(z_4) $);
    \coordinate (z_4_z_1_half) at ($ (z_4)!1/2!(z_1) $); 
    
    \coordinate (p) at ($ (p_pl_1)!1/2!(p_mi_1) $);
    \coordinate (z_1_pl_1) at ($(p_mi_1) + (z_1) - (p_pl_1) $);
    \coordinate (z_3_pl_1) at ($(p_mi_1) + (z_3) - (p_pl_1) $);
    \coordinate (z_2_mi_1) at ($(p_pl_1) + (z_2) - (p_mi_1) $);
    \coordinate (z_4_mi_1) at ($(p_pl_1) + (z_4) - (p_mi_1) $);
    
    \coordinate (z_1_p) at ($ (z_1)!1/2!(z_1_pl_1) $);
    \coordinate (z_3_p) at ($(z_3)!1/2!(z_3_pl_1)$);
    \coordinate (z_2_p) at ($(z_2)!1/2!(z_2_mi_1)$);
    \coordinate (z_4_p) at ($(z_4)!1/2!(z_4_mi_1)$);
        
    \fill[blue!05,opacity=0.3] (3,0,3) -- (3,0,-3) -- (-3,0,-3) -- (-3,0,3) -- cycle;
    \fill[blue!05,opacity=0.3] (3,-5,3) -- (3,-5,-3) -- (-3,-5,-3) -- (-3,-5,3) -- cycle;   
    \fill[blue!05,opacity=0.3] (3,5,3) -- (3,5,-3) -- (-3,5,-3) -- (-3,5,3) -- cycle;
      
    \fill[blue!25,opacity=0.4] (-2,5,-2) -- (-2,5,2) -- (2,-5,2) -- (2,-5,-2) -- cycle;
    
    \fill[blue!25,opacity=0.7] (z_1) -- (z_2_mi_1) -- (z_3) -- (z_4_mi_1) -- cycle;
     
    \fill[blue!25,opacity=0.7] (z_1_pl_1) -- (z_2) -- (z_3_pl_1) -- (z_4) -- cycle;
      
    \fill[blue!25,opacity=0.7] (z_1_p) -- (z_2_p) -- (z_3_p) -- (z_4_p) -- cycle;
      
    \node[fill=blue,circle,inner sep=1pt] at (z_1) {};
    \draw (z_1) node[left] {$\vec z_1$};
    
    \node[fill=blue,circle,inner sep=1pt] at (z_3) {};
    \draw (z_3) node[below] {$\vec z_3$};
    
    \node[fill=blue,circle,inner sep=1pt] at (z_2) {};
    \draw (z_2) node[above] {$\vec z_2$};
    
    \node[fill=blue,circle,inner sep=1pt] at (z_4) {};
    \draw (z_4) node[right] {$\vec z_4$};
 
    \node[fill=blue,circle,inner sep=1pt] at (z_1_pl_1) {};
    \draw (z_1_pl_1) node[right] {$\vec z_1^{R}$};
    
    \node[fill=blue,circle,inner sep=1pt] at (z_3_pl_1) {};
    \draw (z_3_pl_1) node[left] {$\vec z_3^{R}$};
    
    \node[fill=blue,circle,inner sep=1pt] at (z_2_mi_1) {};
    \draw (z_2_mi_1) node[left] {$\vec z_2^{Q}$};
    
    \node[fill=blue,circle,inner sep=1pt] at (z_4_mi_1) {};
    \draw (z_4_mi_1) node[right] {$\vec z_4^{Q}$};
    
    \node[fill=white,circle,inner sep=1pt,draw=blue] at (z_1_p) {};
    \draw (z_1_p) node[right] {$\vec z_1^{\pi}$};
    
    \node[fill=white,circle,inner sep=1pt,draw=blue] at (z_3_p) {};
    \draw (z_3_p) node[left] {$\vec z_3^{\pi}$};
    
    \node[fill=white,circle,inner sep=1pt,draw=blue] at (z_2_p) {};
    \draw (z_2_p) node[below left] {$\vec z_2^{\pi}$};
    
    \node[fill=white,circle,inner sep=1pt,draw=blue] at (z_4_p) {};
    \draw (z_4_p) node[right] {$\vec z_4^{\pi}$};
  
    \node[fill=blue,circle,inner sep=1pt] at (p) {};
    \draw (p) node[right] {$\vec p$};

    \node[fill=white,circle,inner sep=1pt,draw=blue] at (p_pl_1) {};
    \draw (p_pl_1) node[right] {$\vec{p}_{Q}$};
    
    \node[fill=white,circle,inner sep=1pt,draw=blue] at (p_mi_1) {};
    \draw (p_mi_1) node[above right] {$\vec{p}_{R}$};
    
    \node[fill=white,circle,inner sep=1pt,draw=blue] at (z_1_z_2_half) {};
    \draw (z_1_z_2_half) node[left] {$\frac{\vec{z}_1^{\pi} + \vec{z}_2^{\pi}}{2}$};
    
    \node[fill=white,circle,inner sep=1pt,draw=blue] at (z_2_z_3_half) {};
    \draw (z_2_z_3_half) node[below right] {$\frac{\vec{z}_2^{\pi} + \vec{z}_3^{\pi}}{2}$};
    
    \node[fill=white,circle,inner sep=1pt,draw=blue] at (z_3_z_4_half) {};
    \draw (z_3_z_4_half) node[below right] {$\frac{\vec{z}_3^{\pi} + \vec{z}_4^{\pi}}{2}$};
    
    \node[fill=white,circle,inner sep=1pt,draw=blue] at (z_4_z_1_half) {};
    \draw (z_4_z_1_half) node[above] {$\frac{\vec{z}_4^{\pi} + \vec{z}_1^{\pi}}{2}$};
    
    \draw[->] [red,very thin] (z_1) -- (p);
    \draw[->] [red,very thin] (z_1) -- (z_2);
    \draw[->] [red,very thin] (z_1) -- (z_3);
    
    \draw (5, 0, 0) node[right] {$x$};
    \draw (0, 7, 0) node[left] {$z$};
    \draw (0, 0, 7) node[left] {$y$};

    \end{scope}
  \end{tikzpicture}
  \caption{Расположение точек внутри гиперплоскости $S_{1}$ из случая (1) леммы \ref{main_lem}}
\end{figure}

\textbf{А.1.2.} (будет соответствовать утверждению (4))  $\vec{p}^{Q} \in \Z^{4}$, а значит и $\vec{p}^{R} \in \Z^{4}$. Так как $\vec{p} \in \Z^{4}$, то каждая из точек $\vec{z}^{\pi}_{1}$, $\vec{z}^{\pi}_{2}$, $\vec{z}^{\pi}_{3}$, $\vec{z}^{\pi}_{4}$ принадлежит решетке $\Z^{4}$. Тогда
\[(\Delta^{\pi}_{k} \cup \Delta^{Q}_{k} \cup \Delta^{R}_{k}) \cap \Z^{4} = \{\vec{p}, \, \vec{z}_{1}, \, \vec{z}^{Q}_{2}, \, \vec{z}_{3}, \, \vec{z}^{Q}_{4}, \, \vec{z}^{R}_{1}, \, \vec{z}_{2}, \, \vec{z}^{R}_{3}, \, \vec{z}_{4}, \, \vec{p}^{Q}, \, \vec{p}^{R}, \, \vec{z}^{\pi}_{1}, \, \vec{z}^{\pi}_{2}, \, \vec{z}^{\pi}_{3}, \, \vec{z}^{\pi}_{4} \},\]
и набор векторов $\vec{z}_{1}$, $\vec{z}_{2}$, $\vec{p}_{Q} = \frac{1}{2}(\vec{z}_{1}+\vec{z}_{3})$, $\vec{p} = \frac{1}{4}(\vec{z}_{1}+\vec{z}_{2}+\vec{z}_{3}+\vec{z}_{4})$ образует базис решетки $\Z^{4}$, а значит выполняется утверждение (4).

 \begin{figure}[h]
  \centering
  \begin{tikzpicture}[x=10mm, y=7mm, z=-5mm, scale=0.7]
    \begin{scope}[rotate around x=0]
    
    \draw[->] [very thin] (-5,0,0) -- (5,0,0);
    \draw[->] [very thin] (0,-7,0) -- (0,7,0);
    \draw[->] [very thin] (0,0,-7) -- (0,0,7);
    
    \node[fill=black,circle,inner sep=1pt,opacity=0.5] at (0,0,0) {};
    \node[fill=black,circle,inner sep=1pt,opacity=0.5] at (0,0,2) {};
    \node[fill=black,circle,inner sep=1pt,opacity=0.5] at (0,0,-2) {};
    \node[fill=black,circle,inner sep=1pt,opacity=0.5] at (2,0,0) {};
    \node[fill=black,circle,inner sep=1pt,opacity=0.5] at (2,0,2) {};
    \node[fill=black,circle,inner sep=1pt,opacity=0.5] at (2,0,-2) {};
    \node[fill=black,circle,inner sep=1pt,opacity=0.5] at (-2,0,0) {};
    \node[fill=black,circle,inner sep=1pt,opacity=0.5] at (-2,0,2) {};
    \node[fill=black,circle,inner sep=1pt,opacity=0.5] at (-2,0,-2) {};
    \node[fill=black,circle,inner sep=1pt,opacity=0.5] at (0,-5,0) {};
    \node[fill=black,circle,inner sep=1pt,opacity=0.5] at (0,-5,2) {};
    \node[fill=black,circle,inner sep=1pt,opacity=0.5] at (0,-5,-2) {};
    \node[fill=black,circle,inner sep=1pt,opacity=0.5] at (2,-5,0) {};
    \node[fill=black,circle,inner sep=1pt,opacity=0.5] at (2,-5,2) {};
    \node[fill=black,circle,inner sep=1pt,opacity=0.5] at (2,-5,-2) {};
    \node[fill=black,circle,inner sep=1pt,opacity=0.5] at (-2,-5,0) {};
    \node[fill=black,circle,inner sep=1pt,opacity=0.5] at (-2,-5,2) {};
    \node[fill=black,circle,inner sep=1pt,opacity=0.5] at (-2,-5,-2) {};
    \node[fill=black,circle,inner sep=1pt,opacity=0.5] at (2,-5,4) {};
    \node[fill=black,circle,inner sep=1pt,opacity=0.5] at (0,-5,4) {};
    \node[fill=black,circle,inner sep=1pt,opacity=0.5] at (-2,-5,4) {};
    \node[fill=black,circle,inner sep=1pt,opacity=0.5] at (-4,-5,4) {};
    \node[fill=black,circle,inner sep=1pt,opacity=0.5] at (-4,-5,2) {};
    \node[fill=black,circle,inner sep=1pt,opacity=0.5] at (-4,-5,0) {};
    \node[fill=black,circle,inner sep=1pt,opacity=0.5] at (-4,-5,-2) {};    
    \node[fill=black,circle,inner sep=1pt,opacity=0.5] at (0,5,0) {};
    \node[fill=black,circle,inner sep=1pt,opacity=0.5] at (0,5,2) {};
    \node[fill=black,circle,inner sep=1pt,opacity=0.5] at (0,5,-2) {};
    \node[fill=black,circle,inner sep=1pt,opacity=0.5] at (2,5,0) {};
    \node[fill=black,circle,inner sep=1pt,opacity=0.5] at (2,5,2) {};
    \node[fill=black,circle,inner sep=1pt,opacity=0.5] at (2,5,-2) {};
    \node[fill=black,circle,inner sep=1pt,opacity=0.5] at (-2,5,0) {};
    \node[fill=black,circle,inner sep=1pt,opacity=0.5] at (-2,5,2) {};
    \node[fill=black,circle,inner sep=1pt,opacity=0.5] at (-2,5,-2) {};
    \node[fill=black,circle,inner sep=1pt,opacity=0.5] at (4,5,-2) {};
    \node[fill=black,circle,inner sep=1pt,opacity=0.5] at (4,5,2) {};
    \node[fill=black,circle,inner sep=1pt,opacity=0.5] at (4,5,0) {};
    \node[fill=black,circle,inner sep=1pt,opacity=0.5] at (4,5,-4) {};
    \node[fill=black,circle,inner sep=1pt,opacity=0.5] at (0,5,-4) {};
    \node[fill=black,circle,inner sep=1pt,opacity=0.5] at (-2,5,-4) {};
    
    \coordinate (z_1) at (0,5,-2);
    \coordinate (z_3) at (4,5,-2);
    \coordinate (z_2) at (-2,-5,4);
    \coordinate (z_4) at (-2,-5,0);
    
    \coordinate (p_pl_1) at ($ (z_1)!1/2!(z_3) $);
    \coordinate (p_mi_1) at ($ (z_2)!1/2!(z_4) $);
    \coordinate (z_1_z_2_half) at ($ (z_1)!1/2!(z_2) $);
    \coordinate (z_2_z_3_half) at ($ (z_2)!1/2!(z_3) $);
    \coordinate (z_3_z_4_half) at ($ (z_3)!1/2!(z_4) $);
    \coordinate (z_4_z_1_half) at ($ (z_4)!1/2!(z_1) $); 
    
    \coordinate (p) at ($ (p_pl_1)!1/2!(p_mi_1) $);
    \coordinate (z_1_pl_1) at ($(p_mi_1) + (z_1) - (p_pl_1) $);
    \coordinate (z_3_pl_1) at ($(p_mi_1) + (z_3) - (p_pl_1) $);
    \coordinate (z_2_mi_1) at ($(p_pl_1) + (z_2) - (p_mi_1) $);
    \coordinate (z_4_mi_1) at ($(p_pl_1) + (z_4) - (p_mi_1) $);
    
    \coordinate (z_1_p) at ($ (z_1)!1/2!(z_1_pl_1) $);
    \coordinate (z_3_p) at ($(z_3)!1/2!(z_3_pl_1)$);
    \coordinate (z_2_p) at ($(z_2)!1/2!(z_2_mi_1)$);
    \coordinate (z_4_p) at ($(z_4)!1/2!(z_4_mi_1)$);
    
    \fill[blue!05,opacity=0.3] (3,0,3) -- (3,0,-3) -- (-3,0,-3) -- (-3,0,3) -- cycle;
    \fill[blue!05,opacity=0.3] (3,-5,5) -- (3,-5,-3) -- (-5,-5,-3) -- (-5,-5,5) -- cycle;   
    \fill[blue!05,opacity=0.3] (5,5,3) -- (5,5,-5) -- (-3,5,-5) -- (-3,5,3) -- cycle;
     
    \fill[blue!25,opacity=0.7] (z_1) -- (z_2_mi_1) -- (z_3) -- (z_4_mi_1) -- cycle;
     
    \fill[blue!25,opacity=0.7] (z_1_pl_1) -- (z_2) -- (z_3_pl_1) -- (z_4) -- cycle;
      
    \fill[blue!25,opacity=0.7] (z_1_p) -- (z_2_p) -- (z_3_p) -- (z_4_p) -- cycle;
    
    \node[fill=blue,circle,inner sep=1pt] at (z_1) {};
    \draw (z_1) node[left] {$\vec z_1$};
    
    \node[fill=blue,circle,inner sep=1pt] at (z_3) {};
    \draw (z_3) node[right] {$\vec z_3$};
    
    \node[fill=blue,circle,inner sep=1pt] at (z_2) {};
    \draw (z_2) node[left] {$\vec z_2$};
    
    \node[fill=blue,circle,inner sep=1pt] at (z_4) {};
    \draw (z_4) node[right] {$\vec z_4$};

    \node[fill=blue,circle,inner sep=1pt] at (z_1_pl_1) {};
    \draw (z_1_pl_1) node[left] {$\vec z_1^{R}$};
    
    \node[fill=blue,circle,inner sep=1pt] at (z_3_pl_1) {};
    \draw (z_3_pl_1) node[below right] {$\vec z_3^{R}$};
    
    \node[fill=blue,circle,inner sep=1pt] at (z_2_mi_1) {};
    \draw (z_2_mi_1) node[left] {$\vec z_2^{Q}$};
    
    \node[fill=blue,circle,inner sep=1pt] at (z_4_mi_1) {};
    \draw (z_4_mi_1) node[right] {$\vec z_4^{Q}$};
    
     \node[fill=blue,circle,inner sep=1pt] at (z_1_p) {};
    \draw (z_1_p) node[above left] {$\vec z_1^{\pi}$};
    
    \node[fill=blue,circle,inner sep=1pt] at (z_3_p) {};
    \draw (z_3_p) node[below right] {$\vec z_3^{\pi}$};
    
    \node[fill=blue,circle,inner sep=1pt] at (z_2_p) {};
    \draw (z_2_p) node[below left] {$\vec z_2^{\pi}$};
    
    \node[fill=blue,circle,inner sep=1pt] at (z_4_p) {};
    \draw (z_4_p) node[above] {$\vec z_4^{\pi}$};
    
    \node[fill=blue,circle,inner sep=1pt] at (p) {};
    \draw (p) node[right] {$\vec p$};
 
    \node[fill=blue,circle,inner sep=1pt] at (p_pl_1) {};
    \draw (p_pl_1) node[above] {$\vec{p}_{Q}$};
    
    \node[fill=blue,circle,inner sep=1pt,draw=blue] at (p_mi_1) {};
    \draw (p_mi_1) node[right] {$\vec{p}_{R}$};
    
     \node[fill=white,circle,inner sep=1pt,draw=blue] at (z_1_z_2_half) {};
    \draw (z_1_z_2_half) node[left] {$\frac{\vec{z}_1^{\pi} + \vec{z}_2^{\pi}}{2}$};
    
    \node[fill=white,circle,inner sep=1pt,draw=blue] at (z_2_z_3_half) {};
    \draw (z_2_z_3_half) node[below] {$\frac{\vec{z}_2^{\pi} + \vec{z}_3^{\pi}}{2}$};
    
    \node[fill=white,circle,inner sep=1pt,draw=blue] at (z_3_z_4_half) {};
    \draw (z_3_z_4_half) node[right] {$\frac{\vec{z}_3^{\pi} + \vec{z}_4^{\pi}}{2}$};
    
    \node[fill=white,circle,inner sep=1pt,draw=blue] at (z_4_z_1_half) {};
    \draw (z_4_z_1_half) node[left] {$\frac{\vec{z}_4^{\pi} + \vec{z}_1^{\pi}}{2}$};
    
   \draw[->] [red,very thin] (z_1) -- (p);
   \draw[->] [red,very thin] (z_1) -- (z_2);
   \draw[->] [red,very thin] (z_1) -- (p_pl_1);

    \draw (5, 0, 0) node[right] {$x$};
    \draw (0, 7, 0) node[left] {$z$};
    \draw (0, 0, 7) node[left] {$y$};
    
    \end{scope}
  \end{tikzpicture}
  \caption{Расположение точек внутри гиперплоскости $S_{1}$ из случая (4) леммы \ref{main_lem}}
\end{figure}

\textbf{А.2.} Предположим, что $\vec{p} \notin \Z^{4}$. Теперь рассмотрим следующие случаи:

\textbf{А.2.1.} (будет соответствовать утверждению (3)) Каждая из точек  $\frac{\vec{z}^{\pi}_{1} + \vec{z}^{\pi}_{2}}{2}$, $\frac{\vec{z}^{\pi}_{2} + \vec{z}^{\pi}_{3}}{2}$, $\frac{\vec{z}^{\pi}_{3} + \vec{z}^{\pi}_{4}}{2}$, $\frac{\vec{z}^{\pi}_{4} + \vec{z}^{\pi}_{1}}{2}$ принадлежит решетке $\Z^{4}$. В этом случае каждая из точек $\vec{p}^{Q}$, $\vec{p}^{R}$, $ \vec{z}^{R}_{1}$, $ \vec{z}^{R}_{3}$, $\vec{z}^{Q}_{2}$, $\vec{z}^{Q}_{4}$ принадлежит решетке $\Z^{4}$ и никакая из точек $\vec{z}^{\pi}_{1}$, $\vec{z}^{\pi}_{2}$, $\vec{z}^{\pi}_{3}$, $\vec{z}^{\pi}_{4}$ не принадлежит решетке $\Z^{4}$. Тогда
\[(\Delta^{\pi}_{k} \cup \Delta^{Q}_{k} \cup \Delta^{R}_{k}) \cap \Z^{4} = \]
\[\{\vec{z}_{1}, \, \vec{z}^{Q}_{2}, \, \vec{z}_{3}, \, \vec{z}^{Q}_{4}, \, \vec{z}^{R}_{1}, \, \vec{z}_{2}, \, \vec{z}^{R}_{3}, \, \vec{z}_{4}, \, \vec{p}^{Q}, \, \vec{p}^{R}, \, \frac{\vec{z}^{\pi}_{1} + \vec{z}^{\pi}_{2}}{2}, \, \frac{\vec{z}^{\pi}_{2} + \vec{z}^{\pi}_{3}}{2}, \, \frac{\vec{z}^{\pi}_{3} + \vec{z}^{\pi}_{4}}{2}, \, \frac{\vec{z}^{\pi}_{4} + \vec{z}^{\pi}_{1}}{2}\},\]
и набор векторов $\vec{z}_{1}$, $\frac{\vec{z}^{\pi}_{1} + \vec{z}^{\pi}_{2}}{2} = \frac{1}{2}(\vec{z}_{1} + \vec{z}_{2})$, $\vec{p}^{Q} = \frac{1}{2}(\vec{z}_{1}+\vec{z}_{3})$, $\frac{\vec{z}^{\pi}_{4} + \vec{z}^{\pi}_{1}}{2} = \frac{1}{2}(\vec{z}_{1} + \vec{z}_{4})$ образует базис решетки $\Z^{4}$, а значит выполняется утверждение (3).

\begin{figure}[h]
  \centering
  \begin{tikzpicture}[x=10mm, y=7mm, z=-5mm,scale=0.7]
    \begin{scope}[rotate around x=0]

    \draw[->] [very thin] (-5,0,0) -- (5,0,0);
    \draw[->] [very thin] (0,-7,0) -- (0,13,0);
    \draw[->] [very thin] (0,0,-7) -- (0,0,7);

    \node[fill=black,circle,inner sep=1pt,opacity=0.5] at (0,0,0) {};
    \node[fill=black,circle,inner sep=1pt,opacity=0.5] at (0,0,2) {};
    \node[fill=black,circle,inner sep=1pt,opacity=0.5] at (0,0,-2) {};
    \node[fill=black,circle,inner sep=1pt,opacity=0.5] at (2,0,0) {};
    \node[fill=black,circle,inner sep=1pt,opacity=0.5] at (2,0,-2) {};
    \node[fill=black,circle,inner sep=1pt,opacity=0.5] at (-2,0,0) {};
    \node[fill=black,circle,inner sep=1pt,opacity=0.5] at (-2,0,2) {};
    \node[fill=black,circle,inner sep=1pt,opacity=0.5] at (0,-5,0) {};
    \node[fill=black,circle,inner sep=1pt,opacity=0.5] at (0,-5,2) {};
    \node[fill=black,circle,inner sep=1pt,opacity=0.5] at (0,-5,-2) {};
    \node[fill=black,circle,inner sep=1pt,opacity=0.5] at (2,-5,0) {};
    \node[fill=black,circle,inner sep=1pt,opacity=0.5] at (2,-5,2) {};
    \node[fill=black,circle,inner sep=1pt,opacity=0.5] at (2,-5,-2) {};
    \node[fill=black,circle,inner sep=1pt,opacity=0.5] at (-2,-5,0) {};
    \node[fill=black,circle,inner sep=1pt,opacity=0.5] at (-2,-5,2) {};
    \node[fill=black,circle,inner sep=1pt,opacity=0.5] at (-2,-5,-2) {};  
    \node[fill=black,circle,inner sep=1pt,opacity=0.5] at (0,5,0) {};
    \node[fill=black,circle,inner sep=1pt,opacity=0.5] at (0,5,2) {};
    \node[fill=black,circle,inner sep=1pt,opacity=0.5] at (0,5,-2) {};
    \node[fill=black,circle,inner sep=1pt,opacity=0.5] at (2,5,0) {};
    \node[fill=black,circle,inner sep=1pt,opacity=0.5] at (2,5,2) {};
    \node[fill=black,circle,inner sep=1pt,opacity=0.5] at (2,5,-2) {};
    \node[fill=black,circle,inner sep=1pt,opacity=0.5] at (-2,5,0) {};
    \node[fill=black,circle,inner sep=1pt,opacity=0.5] at (-2,5,2) {};
    \node[fill=black,circle,inner sep=1pt,opacity=0.5] at (-2,5,-2) {};    
    \node[fill=black,circle,inner sep=1pt,opacity=0.5] at (0,10,0) {};
    \node[fill=black,circle,inner sep=1pt,opacity=0.5] at (0,10,2) {};
    \node[fill=black,circle,inner sep=1pt,opacity=0.5] at (0,10,-2) {};
    \node[fill=black,circle,inner sep=1pt,opacity=0.5] at (2,10,0) {};
    \node[fill=black,circle,inner sep=1pt,opacity=0.5] at (2,10,2) {};
    \node[fill=black,circle,inner sep=1pt,opacity=0.5] at (2,10,-2) {};
    \node[fill=black,circle,inner sep=1pt,opacity=0.5] at (-2,10,0) {};
    \node[fill=black,circle,inner sep=1pt,opacity=0.5] at (-2,10,2) {};
    \node[fill=black,circle,inner sep=1pt,opacity=0.5] at (-2,10,-2) {};    
    
    \coordinate (z_1) at (0,0,0);
    \coordinate (z_3) at (0,0,4);
    \coordinate (z_2) at (0,10,0);
    \coordinate (z_4) at (4,0,0);
    
    \coordinate (p_pl_1) at ($ (z_1)!1/2!(z_3) $);
    \coordinate (p_mi_1) at ($ (z_2)!1/2!(z_4) $);
    \coordinate (z_1_z_2_half) at ($ (z_1)!1/2!(z_2) $);
    \coordinate (z_2_z_3_half) at ($ (z_2)!1/2!(z_3) $);
    \coordinate (z_3_z_4_half) at ($ (z_3)!1/2!(z_4) $);
    \coordinate (z_4_z_1_half) at ($ (z_4)!1/2!(z_1) $); 
    
    \coordinate (p) at ($ (p_pl_1)!1/2!(p_mi_1) $);
    \coordinate (z_1_pl_1) at ($(p_mi_1) + (z_1) - (p_pl_1) $);
    \coordinate (z_3_pl_1) at ($(p_mi_1) + (z_3) - (p_pl_1) $);
    \coordinate (z_2_mi_1) at ($(p_pl_1) + (z_2) - (p_mi_1) $);
    \coordinate (z_4_mi_1) at ($(p_pl_1) + (z_4) - (p_mi_1) $);
    
    \coordinate (z_1_p) at ($ (z_1)!1/2!(z_1_pl_1) $);
    \coordinate (z_3_p) at ($(z_3)!1/2!(z_3_pl_1)$);
    \coordinate (z_2_p) at ($(z_2)!1/2!(z_2_mi_1)$);
    \coordinate (z_4_p) at ($(z_4)!1/2!(z_4_mi_1)$);    
    
    \fill[blue!05,opacity=0.3] (3,0,3) -- (3,0,-3) -- (-3,0,-3) -- (-3,0,3) -- cycle;
    \fill[blue!05,opacity=0.3] (3,-5,3) -- (3,-5,-3) -- (-3,-5,-3) -- (-3,-5,3) -- cycle;   
    \fill[blue!05,opacity=0.3] (3,5,3) -- (3,5,-3) -- (-3,5,-3) -- (-3,5,3) -- cycle;
    \fill[blue!05,opacity=0.3] (3,10,3) -- (3,10,-3) -- (-3,10,-3) -- (-3,10,3) -- cycle;

    \fill[blue!25,opacity=0.3] (z_1) -- (z_2_mi_1) -- (z_3) -- (z_4_mi_1) -- cycle;
     
    \fill[blue!25,opacity=0.6] (z_1_p) -- (z_2_p) -- (z_3_p) -- (z_4_p) -- cycle;
     
    \fill[blue!25,opacity=0.9] (z_1_pl_1) -- (z_2) -- (z_3_pl_1) -- (z_4) -- cycle;  
    
    \node[fill=blue,circle,inner sep=1pt] at (z_1) {};
    \draw (z_1) node[below right] {$\vec z_1$};
    
    \node[fill=blue,circle,inner sep=1pt] at (z_3) {};
    \draw (z_3) node[left] {$\vec z_3$};
    
    \node[fill=blue,circle,inner sep=1pt] at (z_2) {};
    \draw (z_2) node[above] {$\vec z_2$};
    
     \node[fill=blue,circle,inner sep=1pt] at (z_4) {};
    \draw (z_4) node[above right] {$\vec z_4$};
  
    \node[fill=blue,circle,inner sep=1pt] at (z_1_pl_1) {};
    \draw (z_1_pl_1) node[right] {$\vec z_1^{R}$};
    
    \node[fill=blue,circle,inner sep=1pt] at (z_3_pl_1) {};
    \draw (z_3_pl_1) node[above right] {$\vec z_3^{R}$};
    
    \node[fill=blue,circle,inner sep=1pt] at (z_2_mi_1) {};
    \draw (z_2_mi_1) node[left] {$\vec z_2^{Q}$};
    
    \node[fill=blue,circle,inner sep=1pt] at (z_4_mi_1) {};
    \draw (z_4_mi_1) node[right] {$\vec z_4^{Q}$};
    
    \node[fill=white,circle,inner sep=1pt,draw=blue] at (z_1_p) {};
    \draw (z_1_p) node[below right] {$\vec z_1^{\pi}$};
    
    \node[fill=white,circle,inner sep=1pt,draw=blue] at (z_3_p) {};
    \draw (z_3_p) node[left] {$\vec z_3^{\pi}$};
    
    \node[fill=white,circle,inner sep=1pt,draw=blue] at (z_2_p) {};
    \draw (z_2_p) node[below left] {$\vec z_2^{\pi}$};
    
    \node[fill=white,circle,inner sep=1pt,draw=blue] at (z_4_p) {};
    \draw (z_4_p) node[above right] {$\vec z_4^{\pi}$};

    \node[fill=white,circle,inner sep=1pt,draw=blue] at (p) {};
    \draw (p) node[right] {$\vec p$};

    \node[fill=blue,circle,inner sep=1pt] at (p_pl_1) {};
    \draw (p_pl_1) node[right] {$\vec{p}_{Q}$};
    
    \node[fill=blue,circle,inner sep=1pt] at (p_mi_1) {};
    \draw (p_mi_1) node[above right] {$\vec{p}_{R}$};
    
     \node[fill=blue,circle,inner sep=1pt] at (z_1_z_2_half) {};
    \draw (z_1_z_2_half) node[right] {$\frac{\vec{z}_1^{\pi} + \vec{z}_2^{\pi}}{2}$};
    
    \node[fill=blue,circle,inner sep=1pt] at (z_2_z_3_half) {};
    \draw (z_2_z_3_half) node[left] {$\frac{\vec{z}_2^{\pi} + \vec{z}_3^{\pi}}{2}$};
    
    \node[fill=blue,circle,inner sep=1pt] at (z_3_z_4_half) {};
    \draw (z_3_z_4_half) node[below] {$\frac{\vec{z}_3^{\pi} + \vec{z}_4^{\pi}}{2}$};
    
     \node[fill=blue,circle,inner sep=1pt] at (z_4_z_1_half) {};
    \draw (z_4_z_1_half) node[above right] {$\frac{\vec{z}_4^{\pi} + \vec{z}_1^{\pi}}{2}$};
    
    \draw[->] [red,very thin] (z_1) -- (z_1_z_2_half);
    \draw[->] [red,very thin] (z_1) -- (z_2_z_3_half);
    \draw[->] [red,very thin] (z_1) -- (z_3_z_4_half);

    \draw (5, 0, 0) node[right] {$x$};
    \draw (0, 13, 0) node[left] {$z$};
    \draw (0, 0, 7) node[left] {$y$};

    \end{scope}
  \end{tikzpicture}
  \caption{Расположение точек внутри гиперплоскости $S_{1}$ из случая (3) леммы \ref{main_lem}}
\end{figure}

\textbf{А.2.2.} (будет соответствовать утверждению (7)) Никакая из точек  $\frac{\vec{z}^{\pi}_{1} + \vec{z}^{\pi}_{2}}{2}$, $\frac{\vec{z}^{\pi}_{2} + \vec{z}^{\pi}_{3}}{2}$, $\frac{\vec{z}^{\pi}_{3} + \vec{z}^{\pi}_{4}}{2}$, $\frac{\vec{z}^{\pi}_{4} + \vec{z}^{\pi}_{1}}{2}$ не принадлежит решетке $\Z^{4}$. 

Допустим, что никакая из точек $\vec{z}^{R}_{1}$, $\vec{z}^{R}_{3}$, $\vec{z}^{Q}_{2}$, $\vec{z}^{Q}_{4}$ не принадлежит решетке $\Z^{4}$. Пусть $\vec{r}$ --- некоторый вектор, соединяющий две целые точки из плоскостей $\pi$ и $Q$ соответственно (такой вектор существует, поскольку, по предположению, плоскость $\pi$ --- рациональная). Рассмотрим набор $\vec{z}_{3} - \vec{z}_{1}$, $\vec{z}_{4} - \vec{z}_{2}$, $\vec{z}_{1}$, $\vec{r}$, который является базисом решетки $\Z^{4}$. При этом, первые две координаты точки $\vec{z}_{1}$ в рассматриваемом базисе имеют вид $(0, 0)$, точки $\vec{z}_3$ --- $(1, 0)$, точки $\vec{z}^{Q}_{4}$ --- $(\frac{1}{2}, \frac{1}{2})$, а точки $\vec{z}^{Q}_{2}$ --- $(\frac{1}{2}, -\frac{1}{2})$. Предположим $(z_1, z_2)$ --- первые две координаты точки $\vec{z}_{4}$. Тогда первые две координаты точки $\vec{z}_2$ имеют вид $(z_1, z_{2}-1)$, точки  $\vec{z}^{R}_{1}$ --- $(z_{1} - \frac{1}{2}, z_{2} - \frac{1}{2})$, а точки  $\vec{z}^{R}_{3}$ --- $(z_{1} + \frac{1}{2}, z_{2} - \frac{1}{2})$. Отсюда следует, что первые две координаты точки $\frac{\vec{z}^{\pi}_{1} + \vec{z}^{\pi}_{2}}{2}$ имеют вид $(\frac{z_{1}}{2}, \frac{z_{2} - 1}{2})$, точки $\frac{\vec{z}^{\pi}_{2} + \vec{z}^{\pi}_{3}}{2}$ --- $(\frac{z_{1} + 1}{2}, \frac{z_{2} - 1}{2})$, точки $\frac{\vec{z}^{\pi}_{3} + \vec{z}^{\pi}_{4}}{2}$ --- $(\frac{z_{1} + 1}{2}, \frac{z_{2}}{2})$, а точки $\frac{\vec{z}^{\pi}_{4} + \vec{z}^{\pi}_{1}}{2}$ --- $(\frac{z_{1}}{2}, \frac{z_{2}}{2})$. Из соображений четности получаем противоречие с допущением.

Итак, каждая из точек $\vec{z}^{R}_{1}$, $\vec{z}^{R}_{3}$, $\vec{z}^{Q}_{2}$, $\vec{z}^{Q}_{4}$ принадлежит решетке $\Z^{4}$, а значит каждая из точек $\vec{z}^{\pi}_{1}$, $\vec{z}^{\pi}_{2}$, $\vec{z}^{\pi}_{3}$, $\vec{z}^{\pi}_{4}$ принадлежит решетке $\Z^{4}$, при этом, поскольку $\vec{p} \notin \Z^{4}$, то $\vec{p}^{Q} \notin \Z^4$ и $\vec{p}^{R} \notin \Z^4$. Тогда
\[(\Delta^{\pi}_{k} \cup \Delta^{Q}_{k} \cup \Delta^{R}_{k}) \cap \Z^{4} = \{\vec{z}_{1}, \, \vec{z}^{Q}_{2}, \, \vec{z}_{3}, \, \vec{z}^{Q}_{4}, \, \vec{z}^{R}_{1}, \, \vec{z}_{2}, \, \vec{z}^{R}_{3}, \, \vec{z}_{4}, \, \vec{z}^{\pi}_{1}, \, \vec{z}^{\pi}_{2}, \, \vec{z}^{\pi}_{3}, \, \vec{z}^{\pi}_{4}\},\]
и набор векторов $\vec{z}_{1}$, $\vec{z}_{2}$, $\vec{z}_{3}$, $\vec{z}^{\pi}_{1} = \frac{1}{2}(\vec{z}_{1}+\vec{z}_{2}) + \frac{1}{4}(\vec{z}_{1}+\vec{z}_{4} - \vec{z}_{3} - \vec{z}_{2})$ образует базис решетки $\Z^{4}$, а значит выполняется утверждение (7).

  \begin{figure}[h]
  \centering
  \begin{tikzpicture}[x=10mm, y=7mm, z=-5mm,scale=0.7]
    \begin{scope}[rotate around x=0]

    \draw[->] [very thin] (-5,0,0) -- (5,0,0);
    \draw[->] [very thin] (0,-5,0) -- (0,13,0);
    \draw[->] [very thin] (0,0,-5) -- (0,0,7);

    \node[fill=black,circle,inner sep=1pt,opacity=0.5] at (0,0,0) {};
    \node[fill=black,circle,inner sep=1pt,opacity=0.5] at (0,0,2) {};
    \node[fill=black,circle,inner sep=1pt,opacity=0.5] at (0,0,-2) {};
    \node[fill=black,circle,inner sep=1pt,opacity=0.5] at (2,0,0) {};
    \node[fill=black,circle,inner sep=1pt,opacity=0.5] at (2,0,2) {};
    \node[fill=black,circle,inner sep=1pt,opacity=0.5] at (2,0,-2) {};
    \node[fill=black,circle,inner sep=1pt,opacity=0.5] at (-2,0,0) {};
    \node[fill=black,circle,inner sep=1pt,opacity=0.5] at (-2,0,2) {};
    \node[fill=black,circle,inner sep=1pt,opacity=0.5] at (-2,0,-2) {};
    \node[fill=black,circle,inner sep=1pt,opacity=0.5] at (2,0,4) {};
    \node[fill=black,circle,inner sep=1pt,opacity=0.5] at (0,5,0) {};
    \node[fill=black,circle,inner sep=1pt,opacity=0.5] at (0,5,2) {};
    \node[fill=black,circle,inner sep=1pt,opacity=0.5] at (0,5,-2) {};
    \node[fill=black,circle,inner sep=1pt,opacity=0.5] at (2,5,0) {};
    \node[fill=black,circle,inner sep=1pt,opacity=0.5] at (2,5,2) {};
    \node[fill=black,circle,inner sep=1pt,opacity=0.5] at (2,5,-2) {};
    \node[fill=black,circle,inner sep=1pt,opacity=0.5] at (-2,5,0) {};
    \node[fill=black,circle,inner sep=1pt,opacity=0.5] at (-2,5,2) {};
    \node[fill=black,circle,inner sep=1pt,opacity=0.5] at (-2,5,-2) {};
    \node[fill=black,circle,inner sep=1pt,opacity=0.5] at (0,10,0) {};
    \node[fill=black,circle,inner sep=1pt,opacity=0.5] at (0,10,2) {};
    \node[fill=black,circle,inner sep=1pt,opacity=0.5] at (0,10,-2) {};
    \node[fill=black,circle,inner sep=1pt,opacity=0.5] at (2,10,0) {};
    \node[fill=black,circle,inner sep=1pt,opacity=0.5] at (2,10,2) {};
    \node[fill=black,circle,inner sep=1pt,opacity=0.5] at (2,10,-2) {};
    \node[fill=black,circle,inner sep=1pt,opacity=0.5] at (-2,10,0) {};
    \node[fill=black,circle,inner sep=1pt,opacity=0.5] at (-2,10,2) {};
    \node[fill=black,circle,inner sep=1pt,opacity=0.5] at (-2,10,-2) {};    
    
    \coordinate (z_1) at (2,10,-2);
    \coordinate (z_3) at (4,10,0);
    \coordinate (z_2) at (-2,0,4);
    \coordinate (z_4) at (0,0,2);
    
    \coordinate (p_pl_1) at ($ (z_1)!1/2!(z_3) $);
    \coordinate (p_mi_1) at ($ (z_2)!1/2!(z_4) $);
    \coordinate (z_1_z_2_half) at ($ (z_1)!1/2!(z_2) $);
    \coordinate (z_2_z_3_half) at ($ (z_2)!1/2!(z_3) $);
    \coordinate (z_3_z_4_half) at ($ (z_3)!1/2!(z_4) $);
    \coordinate (z_4_z_1_half) at ($ (z_4)!1/2!(z_1) $); 
    
    \coordinate (p) at ($ (p_pl_1)!1/2!(p_mi_1) $);
    \coordinate (z_1_pl_1) at ($(p_mi_1) + (z_1) - (p_pl_1) $);
    \coordinate (z_3_pl_1) at ($(p_mi_1) + (z_3) - (p_pl_1) $);
    \coordinate (z_2_mi_1) at ($(p_pl_1) + (z_2) - (p_mi_1) $);
    \coordinate (z_4_mi_1) at ($(p_pl_1) + (z_4) - (p_mi_1) $);
    
    \coordinate (z_1_p) at ($ (z_1)!1/2!(z_1_pl_1) $);
    \coordinate (z_3_p) at ($(z_3)!1/2!(z_3_pl_1)$);
    \coordinate (z_2_p) at ($(z_2)!1/2!(z_2_mi_1)$);
    \coordinate (z_4_p) at ($(z_4)!1/2!(z_4_mi_1)$);
     
    \fill[blue!05,opacity=0.3] (3,0,5) -- (3,0,-3) -- (-3,0,-3) -- (-3,0,5) -- cycle;
    \fill[blue!05,opacity=0.3] (3,5,3) -- (3,5,-3) -- (-3,5,-3) -- (-3,5,3) -- cycle;
    \fill[blue!05,opacity=0.3] (3,10,3) -- (3,10,-3) -- (-3,10,-3) -- (-3,10,3) -- cycle;   

    \fill[blue!25,opacity=0.6] (z_1) -- (z_2_mi_1) -- (z_3) -- (z_4_mi_1) -- cycle;
     
    \fill[blue!25,opacity=0.6] (z_1_p) -- (z_2_p) -- (z_3_p) -- (z_4_p) -- cycle;
     
    \fill[blue!25,opacity=0.6] (z_1_pl_1) -- (z_2) -- (z_3_pl_1) -- (z_4) -- cycle;  
    
    \node[fill=blue,circle,inner sep=1pt] at (z_1) {};
    \draw (z_1) node[above] {$\vec z_1$};
    
    \node[fill=blue,circle,inner sep=1pt] at (z_3) {};
    \draw (z_3) node[right] {$\vec z_3$};
    
    \node[fill=blue,circle,inner sep=1pt] at (z_2) {};
    \draw (z_2) node[left] {$\vec z_2$};
    
     \node[fill=blue,circle,inner sep=1pt] at (z_4) {};
    \draw (z_4) node[right] {$\vec z_4$};
  
    \node[fill=blue,circle,inner sep=1pt] at (z_1_pl_1) {};
    \draw (z_1_pl_1) node[above] {$\vec z_1^{R}$};
    
    \node[fill=blue,circle,inner sep=1pt] at (z_3_pl_1) {};
    \draw (z_3_pl_1) node[right] {$\vec z_3^{R}$};
    
    \node[fill=blue,circle,inner sep=1pt] at (z_2_mi_1) {};
    \draw (z_2_mi_1) node[left] {$\vec z_2^{Q}$};
    
    \node[fill=blue,circle,inner sep=1pt] at (z_4_mi_1) {};
    \draw (z_4_mi_1) node[right] {$\vec z_4^{Q}$};
    
    \node[fill=blue,circle,inner sep=1pt] at (z_1_p) {};
    \draw (z_1_p) node[left] {$\vec z_1^{\pi}$};
    
    \node[fill=blue,circle,inner sep=1pt] at (z_3_p) {};
    \draw (z_3_p) node[below right] {$\vec z_3^{\pi}$};
    
    \node[fill=blue,circle,inner sep=1pt] at (z_2_p) {};
    \draw (z_2_p) node[below left] {$\vec z_2^{\pi}$};
    
    \node[fill=blue,circle,inner sep=1pt] at (z_4_p) {};
    \draw (z_4_p) node[above right] {$\vec z_4^{\pi}$};

    \node[fill=white,circle,inner sep=1pt,draw=blue] at (p) {};
    \draw (p) node[right] {$\vec p$};

    \node[fill=white,circle,inner sep=1pt,draw=blue] at (p_pl_1) {};
    \draw (p_pl_1) node[right] {$\vec{p}_{Q}$};
    
    \node[fill=white,circle,inner sep=1pt,draw=blue] at (p_mi_1) {};
    \draw (p_mi_1) node[above] {$\vec{p}_{R}$};
    
    \node[fill=white,circle,inner sep=1pt,draw=blue] at (z_1_z_2_half) {};
    \draw (z_1_z_2_half) node[left] {$\frac{\vec{z}_1^{\pi} + \vec{z}_2^{\pi}}{2}$};
    
    \node[fill=white,circle,inner sep=1pt,draw=blue] at (z_2_z_3_half) {};
    \draw (z_2_z_3_half) node[below] {$\frac{\vec{z}_2^{\pi} + \vec{z}_3^{\pi}}{2}$};
    
    \node[fill=white,circle,inner sep=1pt,draw=blue] at (z_3_z_4_half) {};
    \draw (z_3_z_4_half) node[right] {$\frac{\vec{z}_3^{\pi} + \vec{z}_4^{\pi}}{2}$};
    
    \node[fill=white,circle,inner sep=1pt,draw=blue] at (z_4_z_1_half) {};
    \draw (z_4_z_1_half) node[above] {$\frac{\vec{z}_4^{\pi} + \vec{z}_1^{\pi}}{2}$};
    
    \draw[->] [red,very thin] (z_1) -- (z_3);
    \draw[->] [red,very thin] (z_1) -- (z_2);
    \draw[->] [red,very thin] (z_1) -- (z_1_p);

    \draw (5, 0, 0) node[right] {$x$};
    \draw (0, 13, 0) node[left] {$z$};
    \draw (0, 0, 7) node[left] {$y$};

    \end{scope}
  \end{tikzpicture}
  \caption{Расположение точек внутри гиперплоскости $S_{1}$ из случая (7) леммы \ref{main_lem}}
\end{figure}

\textbf{Б.} Предположим, что $\pi$ не является рациональной плоскостью. Покажем, что этот случай разбивается на 3 принципиально разных случая.

\textbf{Б.1.} (будет соответствовать утверждению (5)) Предположим, что  $\vec{p}^{Q} \in \Z^{4}$, а значит и $\vec{p}^{R} \in \Z^{4}$. Тогда 
\[(\Delta^{\pi}_{k} \cup \Delta^{Q}_{k} \cup \Delta^{R}_{k}) \cap \Z^{4} = \{\vec{z}_{1}, \, \vec{z}^{Q}_{2}, \, \vec{z}_{3}, \, \vec{z}^{Q}_{4}, \, \vec{z}^{R}_{1}, \, \vec{z}_{2}, \, \vec{z}^{R}_{3}, \, \vec{z}_{4}, \, \vec{p}^{Q}, \, \vec{p}^{R}\},\]
и набор векторов $\vec{z}_{1}$, $\vec{z}_{2}$, $\vec{p}^{Q} = \frac{1}{2}(\vec{z}_{1}+\vec{z}_{3})$, $\vec{p}^{R} = \frac{1}{2}(\vec{z}_{2}+\vec{z}_{4})$ образует базис решетки $\Z^{4}$, а значит выполняется утверждение (5).

  \begin{figure}[h]
  \centering
  \begin{tikzpicture}[x=10mm, y=7mm, z=-5mm,scale=0.7]
    \begin{scope}[rotate around x=0] 

    \draw[->] [very thin] (-5,0,0) -- (5,0,0);
    \draw[->] [very thin] (0,-4,0) -- (0,7,0);
    \draw[->] [very thin] (0,0,-5) -- (0,0,6);    
    
    \node[fill=black,circle,inner sep=1pt,opacity=0.5] at (0,0,0) {};
    \node[fill=black,circle,inner sep=1pt,opacity=0.5] at (0,0,2) {};
    \node[fill=black,circle,inner sep=1pt,opacity=0.5] at (0,0,-2) {};
    \node[fill=black,circle,inner sep=1pt,opacity=0.5] at (2,0,0) {};
    \node[fill=black,circle,inner sep=1pt,opacity=0.5] at (2,0,2) {};
    \node[fill=black,circle,inner sep=1pt,opacity=0.5] at (2,0,-2) {};
    \node[fill=black,circle,inner sep=1pt,opacity=0.5] at (-2,0,0) {};
    \node[fill=black,circle,inner sep=1pt,opacity=0.5] at (-2,0,2) {};
    \node[fill=black,circle,inner sep=1pt,opacity=0.5] at (-2,0,-2) {};
    \node[fill=black,circle,inner sep=1pt,opacity=0.5] at (2,0,4) {};  
    \node[fill=black,circle,inner sep=1pt,opacity=0.5] at (0,5,0) {};
    \node[fill=black,circle,inner sep=1pt,opacity=0.5] at (0,5,2) {};
    \node[fill=black,circle,inner sep=1pt,opacity=0.5] at (0,5,-2) {};
    \node[fill=black,circle,inner sep=1pt,opacity=0.5] at (2,5,0) {};
    \node[fill=black,circle,inner sep=1pt,opacity=0.5] at (2,5,2) {};
    \node[fill=black,circle,inner sep=1pt,opacity=0.5] at (2,5,-2) {};
    \node[fill=black,circle,inner sep=1pt,opacity=0.5] at (-2,5,0) {};
    \node[fill=black,circle,inner sep=1pt,opacity=0.5] at (-2,5,2) {};
    \node[fill=black,circle,inner sep=1pt,opacity=0.5] at (-2,5,-2) {};
    \node[fill=black,circle,inner sep=1pt,opacity=0.5] at (4,5,2) {};
    \node[fill=black,circle,inner sep=1pt,opacity=0.5] at (4,5,0) {};
    \node[fill=black,circle,inner sep=1pt,opacity=0.5] at (4,5,-2) {};
     
    \coordinate (z_1) at (0,0,0);
    \coordinate (z_3) at (0,0,4);
    \coordinate (z_2) at (0,5,0);
    \coordinate (z_4) at (4,5,0);
    
    \coordinate (p_pl_1) at ($ (z_1)!1/2!(z_3) $);
    \coordinate (p_mi_1) at ($ (z_2)!1/2!(z_4) $);
    \coordinate (z_1_z_2_half) at ($ (z_1)!1/2!(z_2) $);
    \coordinate (z_2_z_3_half) at ($ (z_2)!1/2!(z_3) $);
    \coordinate (z_3_z_4_half) at ($ (z_3)!1/2!(z_4) $);
    \coordinate (z_4_z_1_half) at ($ (z_4)!1/2!(z_1) $); 
    
    \coordinate (p) at ($ (p_pl_1)!1/2!(p_mi_1) $);
    \coordinate (z_1_pl_1) at ($(p_mi_1) + (z_1) - (p_pl_1) $);
    \coordinate (z_3_pl_1) at ($(p_mi_1) + (z_3) - (p_pl_1) $);
    \coordinate (z_2_mi_1) at ($(p_pl_1) + (z_2) - (p_mi_1) $);
    \coordinate (z_4_mi_1) at ($(p_pl_1) + (z_4) - (p_mi_1) $);
    
    \coordinate (z_1_p) at ($ (z_1)!1/2!(z_1_pl_1) $);
    \coordinate (z_3_p) at ($(z_3)!1/2!(z_3_pl_1)$);
    \coordinate (z_2_p) at ($(z_2)!1/2!(z_2_mi_1)$);
    \coordinate (z_4_p) at ($(z_4)!1/2!(z_4_mi_1)$);  
    
    \fill[blue!05,opacity=0.3] (3,0,5) -- (3,0,-3) -- (-3,0,-3) -- (-3,0,5) -- cycle;
    \fill[blue!05,opacity=0.3] (5,5,3) -- (5,5,-3) -- (-3,5,-3) -- (-3,5,3) -- cycle;

    \fill[blue!25,opacity=0.3] (z_1) -- (z_2_mi_1) -- (z_3) -- (z_4_mi_1) -- cycle;

    \fill[blue!25,opacity=0.9] (z_1_pl_1) -- (z_2) -- (z_3_pl_1) -- (z_4) -- cycle;
         
    \node[fill=blue,circle,inner sep=1pt] at (z_1) {};
    \draw (z_1) node[above left] {$\vec z_1$};
    
    \node[fill=blue,circle,inner sep=1pt] at (z_3) {};
    \draw (z_3) node[right] {$\vec z_3$};
    
    \node[fill=blue,circle,inner sep=1pt] at (z_2) {};
    \draw (z_2) node[left] {$\vec z_2$};
    
     \node[fill=blue,circle,inner sep=1pt] at (z_4) {};
    \draw (z_4) node[right] {$\vec z_4$};
  
    \node[fill=blue,circle,inner sep=1pt] at (z_1_pl_1) {};
    \draw (z_1_pl_1) node[above] {$\vec z_1^{R}$};
    
    \node[fill=blue,circle,inner sep=1pt] at (z_3_pl_1) {};
    \draw (z_3_pl_1) node[left] {$\vec z_3^{R}$};
    
    \node[fill=blue,circle,inner sep=1pt] at (z_2_mi_1) {};
    \draw (z_2_mi_1) node[left] {$\vec z_2^{Q}$};
    
    \node[fill=blue,circle,inner sep=1pt] at (z_4_mi_1) {};
    \draw (z_4_mi_1) node[right] {$\vec z_4^{Q}$};

    \node[fill=white,circle,inner sep=1pt,draw=blue] at (p) {};
    \draw (p) node[left] {$\vec p$};

    \node[fill=blue,circle,inner sep=1pt,draw=blue] at (p_pl_1) {};
    \draw (p_pl_1) node[below right] {$\vec{p}_{Q}$};
    
    \node[fill=blue,circle,inner sep=1pt,draw=blue] at (p_mi_1) {};
    \draw (p_mi_1) node[above] {$\vec{p}_{R}$};
    
    \draw[->] [red,very thin] (z_1) -- (z_2);
    \draw[->] [red,very thin] (z_1) -- (p_pl_1);
    \draw[->] [red,very thin] (z_1) -- (p_mi_1);

    \draw (5, 0, 0) node[right] {$x$};
    \draw (0, 7, 0) node[left] {$z$};
    \draw (0, 0, 6) node[left] {$y$};

    \end{scope}
  \end{tikzpicture}
  \caption{Расположение точек внутри гиперплоскости $S_{1}$ из случая (5) леммы \ref{main_lem}}
\end{figure}

\textbf{Б.2.} Предположим, что $\vec{p}^{Q} \notin \Z^{4}$, а значит и $\vec{p}^{R} \notin \Z^{4}$. Это предположение дает нам два последних случая.

\textbf{Б.2.1.} (будет соответствовать утверждению (6)) Предположим, что $\vec{z}^{Q}_{2} \in \Z^{4}$, $\vec{z}^{Q}_{4} \in \Z^{4}$, $\vec{z}^{R}_{1} \in \Z^{4}$ и $\vec{z}^{R}_{2} \in \Z^{4}$. Тогда 
\[(\Delta^{\pi}_{k} \cup \Delta^{Q}_{k} \cup \Delta^{R}_{k}) \cap \Z^{4} = \{\vec{z}_{1}, \, \vec{z}^{Q}_{2}, \, \vec{z}_{3}, \, \vec{z}^{Q}_{4}, \, \vec{z}^{R}_{1}, \, \vec{z}_{2}, \, \vec{z}^{R}_{3}, \, \vec{z}_{4}\},\]
и набор векторов $\vec{z}_{1}$, $\vec{z}_{2}$, $\vec{z}_{3}$, $\vec{z}^{Q}_{4} = \frac{1}{2}(\vec{z}_{1} + \vec{z}_{3} + \vec{z}_{4} - \vec{z}_{2})$ образует базис решетки $\Z^{4}$, а значит выполняется утверждение (6).

 \begin{figure}[h]
  \centering
  \begin{tikzpicture}[x=10mm, y=7mm, z=-5mm,scale=0.7]
    \begin{scope}[rotate around x=0]
  
    \draw[->] [very thin] (-5,0,0) -- (5,0,0);
    \draw[->] [very thin] (0,-3,0) -- (0,7,0);
    \draw[->] [very thin] (0,0,-4) -- (0,0,4);   
    
    \node[fill=black,circle,inner sep=1pt,opacity=0.5] at (0,0,0) {};
    \node[fill=black,circle,inner sep=1pt,opacity=0.5] at (0,0,2) {};
    \node[fill=black,circle,inner sep=1pt,opacity=0.5] at (0,0,-2) {};
    \node[fill=black,circle,inner sep=1pt,opacity=0.5] at (2,0,0) {};
    \node[fill=black,circle,inner sep=1pt,opacity=0.5] at (2,0,2) {};
    \node[fill=black,circle,inner sep=1pt,opacity=0.5] at (2,0,-2) {};
    \node[fill=black,circle,inner sep=1pt,opacity=0.5] at (-2,0,0) {};
    \node[fill=black,circle,inner sep=1pt,opacity=0.5] at (-2,0,2) {};
    \node[fill=black,circle,inner sep=1pt,opacity=0.5] at (-2,0,-2) {};
    \node[fill=black,circle,inner sep=1pt,opacity=0.5] at (0,5,0) {};
    \node[fill=black,circle,inner sep=1pt,opacity=0.5] at (0,5,2) {};
    \node[fill=black,circle,inner sep=1pt,opacity=0.5] at (0,5,-2) {};
    \node[fill=black,circle,inner sep=1pt,opacity=0.5] at (2,5,0) {};
    \node[fill=black,circle,inner sep=1pt,opacity=0.5] at (2,5,2) {};
    \node[fill=black,circle,inner sep=1pt,opacity=0.5] at (2,5,-2) {};
    \node[fill=black,circle,inner sep=1pt,opacity=0.5] at (-2,5,0) {};
    \node[fill=black,circle,inner sep=1pt,opacity=0.5] at (-2,5,2) {};
    \node[fill=black,circle,inner sep=1pt,opacity=0.5] at (-2,5,-2) {};
    
    \coordinate (z_1) at (2,5,0);
    \coordinate (z_3) at (0,5,-2);
    \coordinate (z_2) at (0,0,0);
    \coordinate (z_4) at (-2,0,2);
    
    \coordinate (p_pl_1) at ($ (z_1)!1/2!(z_3) $);
    \coordinate (p_mi_1) at ($ (z_2)!1/2!(z_4) $);
    \coordinate (z_1_z_2_half) at ($ (z_1)!1/2!(z_2) $);
    \coordinate (z_2_z_3_half) at ($ (z_2)!1/2!(z_3) $);
    \coordinate (z_3_z_4_half) at ($ (z_3)!1/2!(z_4) $);
    \coordinate (z_4_z_1_half) at ($ (z_4)!1/2!(z_1) $); 
    
    \coordinate (p) at ($ (p_pl_1)!1/2!(p_mi_1) $);
    \coordinate (z_1_pl_1) at ($(p_mi_1) + (z_1) - (p_pl_1) $);
    \coordinate (z_3_pl_1) at ($(p_mi_1) + (z_3) - (p_pl_1) $);
    \coordinate (z_2_mi_1) at ($(p_pl_1) + (z_2) - (p_mi_1) $);
    \coordinate (z_4_mi_1) at ($(p_pl_1) + (z_4) - (p_mi_1) $);
    
    \coordinate (z_1_p) at ($ (z_1)!1/2!(z_1_pl_1) $);
    \coordinate (z_3_p) at ($(z_3)!1/2!(z_3_pl_1)$);
    \coordinate (z_2_p) at ($(z_2)!1/2!(z_2_mi_1)$);
    \coordinate (z_4_p) at ($(z_4)!1/2!(z_4_mi_1)$);   
    
    \fill[blue!05,opacity=0.3] (3,0,3) -- (3,0,-3) -- (-3,0,-3) -- (-3,0,3) -- cycle;
    \fill[blue!05,opacity=0.3] (3,5,3) -- (3,5,-3) -- (-3,5,-3) -- (-3,5,3) -- cycle;   

    \fill[blue!25,opacity=0.6] (z_1) -- (z_2_mi_1) -- (z_3) -- (z_4_mi_1) -- cycle;
     
    \fill[blue!25,opacity=0.6] (z_1_pl_1) -- (z_2) -- (z_3_pl_1) -- (z_4) -- cycle;
      
    \node[fill=blue,circle,inner sep=1pt] at (z_1) {};
    \draw (z_1) node[below right] {$\vec z_1$};
    
    \node[fill=blue,circle,inner sep=1pt] at (z_3) {};
    \draw (z_3) node[above] {$\vec z_3$};
    
    \node[fill=blue,circle,inner sep=1pt] at (z_2) {};
    \draw (z_2) node[below right] {$\vec z_2$};
    
     \node[fill=blue,circle,inner sep=1pt] at (z_4) {};
    \draw (z_4) node[left] {$\vec z_4$};
  
    \node[fill=blue,circle,inner sep=1pt] at (z_1_pl_1) {};
    \draw (z_1_pl_1) node[right] {$\vec z_1^{R}$};
    
    \node[fill=blue,circle,inner sep=1pt] at (z_3_pl_1) {};
    \draw (z_3_pl_1) node[above right] {$\vec z_3^{R}$};
    
    \node[fill=blue,circle,inner sep=1pt] at (z_2_mi_1) {};
    \draw (z_2_mi_1) node[above right] {$\vec z_2^{Q}$};
    
    \node[fill=blue,circle,inner sep=1pt] at (z_4_mi_1) {};
    \draw (z_4_mi_1) node[left] {$\vec z_4^{Q}$};

    \node[fill=white,circle,inner sep=1pt,draw=blue] at (p) {};
    \draw (p) node[right] {$\vec p$};
  
    \node[fill=white,circle,inner sep=1pt,draw=blue] at (p_pl_1) {};
    \draw (p_pl_1) node[above right] {$\vec{p}_{Q}$};
    
    \node[fill=white,circle,inner sep=1pt,draw=blue] at (p_mi_1) {};
    \draw (p_mi_1) node[right] {$\vec{p}_{R}$};
  
    \draw[->] [red,very thin] (z_1) -- (z_2);
    \draw[->] [red,very thin] (z_1) -- (z_3);
    \draw[->] [red,very thin] (z_1) -- (z_2_mi_1);

    \draw (5, 0, 0) node[right] {$x$};
    \draw (0, 7, 0) node[left] {$z$};
    \draw (0, 0, 4) node[left] {$y$};

    \end{scope}
  \end{tikzpicture}
  \caption{Расположение точек внутри гиперплоскости $S_{1}$ из случая (6) леммы \ref{main_lem}}
\end{figure}

\textbf{Б.2.2.} (будет соответствовать утверждению (2)) Предположим, что $\vec{z}^{Q}_{2} \notin \Z^{4}$, $\vec{z}^{Q}_{4} \notin \Z^{4}$, $\vec{z}^{R}_{1} \notin \Z^{4}$ и $\vec{z}^{R}_{2} \notin \Z^{4}$. 
Тогда 
\[(\Delta^{\pi}_{k} \cup \Delta^{Q}_{k} \cup \Delta^{R}_{k}) \cap \Z^{4} = \{\vec{z}_{1}, \, \vec{z}_{3}, \, \vec{z}_{2}, \, \vec{z}_{4}\},\]
и набор векторов $\vec{z}_{1}$, $\vec{z}_{2}$, $\vec{z}_{3}$, $\vec{z}_{4}$ образует базис решетки $\Z^{4}$, а значит выполняется утверждение (2), что завершает доказательство леммы

 \begin{figure}[h]
  \centering
  \begin{tikzpicture}[x=10mm, y=7mm, z=-5mm,scale=0.7]
    \begin{scope}[rotate around x=0]

    \draw[->] [very thin] (-5,0,0) -- (5,0,0);
    \draw[->] [very thin] (0,-3,0) -- (0,7,0);
    \draw[->] [very thin] (0,0,-4) -- (0,0,4);   
    
    \node[fill=black,circle,inner sep=1pt,opacity=0.5] at (0,0,0) {};
    \node[fill=black,circle,inner sep=1pt,opacity=0.5] at (0,0,2) {};
    \node[fill=black,circle,inner sep=1pt,opacity=0.5] at (0,0,-2) {};
    \node[fill=black,circle,inner sep=1pt,opacity=0.5] at (2,0,0) {};
    \node[fill=black,circle,inner sep=1pt,opacity=0.5] at (2,0,2) {};
    \node[fill=black,circle,inner sep=1pt,opacity=0.5] at (2,0,-2) {};
    \node[fill=black,circle,inner sep=1pt,opacity=0.5] at (-2,0,0) {};
    \node[fill=black,circle,inner sep=1pt,opacity=0.5] at (-2,0,2) {};
    \node[fill=black,circle,inner sep=1pt,opacity=0.5] at (-2,0,-2) {};
    \node[fill=black,circle,inner sep=1pt,opacity=0.5] at (0,5,0) {};
    \node[fill=black,circle,inner sep=1pt,opacity=0.5] at (0,5,2) {};
    \node[fill=black,circle,inner sep=1pt,opacity=0.5] at (0,5,-2) {};
    \node[fill=black,circle,inner sep=1pt,opacity=0.5] at (2,5,0) {};
    \node[fill=black,circle,inner sep=1pt,opacity=0.5] at (2,5,2) {};
    \node[fill=black,circle,inner sep=1pt,opacity=0.5] at (2,5,-2) {};
    \node[fill=black,circle,inner sep=1pt,opacity=0.5] at (-2,5,0) {};
    \node[fill=black,circle,inner sep=1pt,opacity=0.5] at (-2,5,2) {};
    \node[fill=black,circle,inner sep=1pt,opacity=0.5] at (-2,5,-2) {};
    
    \coordinate (z_1) at (0,0,0);
    \coordinate (z_3) at (0,0,2);
    \coordinate (z_2) at (0,5,0);
    \coordinate (z_4) at (2,0,0);
    
    \coordinate (p_pl_1) at ($ (z_1)!1/2!(z_3) $);
    \coordinate (p_mi_1) at ($ (z_2)!1/2!(z_4) $);
    \coordinate (z_1_z_2_half) at ($ (z_1)!1/2!(z_2) $);
    \coordinate (z_2_z_3_half) at ($ (z_2)!1/2!(z_3) $);
    \coordinate (z_3_z_4_half) at ($ (z_3)!1/2!(z_4) $);
    \coordinate (z_4_z_1_half) at ($ (z_4)!1/2!(z_1) $); 
    
    \coordinate (p) at ($ (p_pl_1)!1/2!(p_mi_1) $);
    \coordinate (z_1_pl_1) at ($(p_mi_1) + (z_1) - (p_pl_1) $);
    \coordinate (z_3_pl_1) at ($(p_mi_1) + (z_3) - (p_pl_1) $);
    \coordinate (z_2_mi_1) at ($(p_pl_1) + (z_2) - (p_mi_1) $);
    \coordinate (z_4_mi_1) at ($(p_pl_1) + (z_4) - (p_mi_1) $);
    
    \coordinate (z_1_p) at ($ (z_1)!1/2!(z_1_pl_1) $);
    \coordinate (z_3_p) at ($(z_3)!1/2!(z_3_pl_1)$);
    \coordinate (z_2_p) at ($(z_2)!1/2!(z_2_mi_1)$);
    \coordinate (z_4_p) at ($(z_4)!1/2!(z_4_mi_1)$);

    \fill[blue!05,opacity=0.3] (3,0,3) -- (3,0,-3) -- (-3,0,-3) -- (-3,0,3) -- cycle;
    \fill[blue!05,opacity=0.3] (3,5,3) -- (3,5,-3) -- (-3,5,-3) -- (-3,5,3) -- cycle;
    
    \fill[blue!25,opacity=0.6] (z_1) -- (z_2_mi_1) -- (z_3) -- (z_4_mi_1) -- cycle;
     
    \fill[blue!25,opacity=0.6] (z_1_pl_1) -- (z_2) -- (z_3_pl_1) -- (z_4) -- cycle;

    \node[fill=blue,circle,inner sep=1pt] at (z_1) {};
    \draw (z_1) node[below right] {$\vec z_1$};
    
    \node[fill=blue,circle,inner sep=1pt] at (z_3) {};
    \draw (z_3) node[left] {$\vec z_3$};
    
    \node[fill=blue,circle,inner sep=1pt] at (z_2) {};
    \draw (z_2) node[left] {$\vec z_2$};
    
     \node[fill=blue,circle,inner sep=1pt] at (z_4) {};
    \draw (z_4) node[above right] {$\vec z_4$};
  
    \node[fill=white,circle,inner sep=1pt,draw=blue] at (z_1_pl_1) {};
    \draw (z_1_pl_1) node[above right] {$\vec z_1^{R}$};
    
    \node[fill=white,circle,inner sep=1pt,draw=blue] at (z_3_pl_1) {};
    \draw (z_3_pl_1) node[above] {$\vec z_3^{R}$};
    
    \node[fill=white,circle,inner sep=1pt,draw=blue] at (z_2_mi_1) {};
    \draw (z_2_mi_1) node[left] {$\vec z_2^{Q}$};
    
    \node[fill=white,circle,inner sep=1pt, draw=blue] at (z_4_mi_1) {};
    \draw (z_4_mi_1) node[left] {$\vec z_4^{Q}$};
   
    \node[fill=white,circle,inner sep=1pt, draw=blue] at (p) {};
    \draw (p) node[right] {$\vec p$};
    
    \node[fill=white,circle,inner sep=1pt,draw=blue] at (p_pl_1) {};
    \draw (p_pl_1) node[above] {$\vec{p}_{Q}$};
    
    \node[fill=white,circle,inner sep=1pt,draw=blue] at (p_mi_1) {};
    \draw (p_mi_1) node[right] {$\vec{p}_{R}$};
    
    \draw[->] [red,very thin] (z_1) -- (z_2);
    \draw[->] [red,very thin] (z_1) -- (z_3);
    \draw[->] [red,very thin] (z_1) -- (z_4);

    \draw (5, 0, 0) node[right] {$x$};
    \draw (0, 7, 0) node[left] {$z$};
    \draw (0, 0, 4) node[left] {$y$};

    \end{scope}
  \end{tikzpicture}
  \caption{Расположение точек внутри гиперплоскости $S_{1}$ из случая (2) леммы \ref{main_lem}}
\end{figure}
\end{proof}

 Если задана дробь $\cf(l_1,l_2,l_3,l_4)=\cf(A)\in\gA_3$, будем считать, что подпространство $l_1$ порождается вектором $\vec l_1=(1,\alpha,\beta,\gamma)$. Тогда из предложения \ref{prop:more_than_pelle_n_dim} следует, что числа $1,\alpha,\beta, \gamma$ образуют базис поля $K=\Q(\alpha,\beta, \gamma)$ над $\Q$ и каждое $l_i$ порождается вектором $\vec l_i=(1,\sigma_i(\alpha),\sigma_i(\beta), \sigma_i(\gamma))$, где $\sigma_1(=\id),\sigma_2,\sigma_3, \sigma_4$ --- все вложения $K$ в $\R$. То есть, если через $\big(\vec l_1,\vec l_2,\vec l_3,\vec l_4\big)$ обозначить матрицу со столбцами $\vec l_1,\vec l_2,\vec l_3,\vec l_4$, получим
 \[
  \big(\vec l_1,\vec l_2,\vec l_3,\vec l_4\big)=
  \begin{pmatrix}
    1 & 1 & 1 & 1 \\
    \alpha & \sigma_2(\alpha) & \sigma_3(\alpha) & \sigma_4(\alpha) \\
    \beta & \sigma_2(\beta) & \sigma_3(\beta) & \sigma_4(\beta) \\
    \gamma & \sigma_2(\gamma) & \sigma_3(\gamma) & \sigma_4(\gamma)
  \end{pmatrix}.
\]

Напомним, что через $\gA_{3}'$ мы обозначали множество всех трехмерных алгебраических цепных дробей, для которых поле $K$ из предложения \ref{prop:more_than_pelle_n_dim} --- вполне вещественное циклическое расширение Галуа. Пусть $\sigma$ --- образующая группы Галуа $\gal(K/\Q)$. Также мы выбирали такую нумерацию прямых $l_1, l_2, l_3, l_4$, что если через $\big(\vec l_1,\vec l_2, \vec l_3, \vec l_4 \big)$ обозначить матрицу со столбцами $\vec l_1,\vec l_2, \vec l_3, \vec l_4$, то 
 \[
  \big(\vec l_1,\vec l_2, \vec l_3, \vec l_4 \big)=
  \begin{pmatrix}
     1 & 1 & 1 & 1 \\
    \alpha & \sigma(\alpha) & \sigma^{2}(\alpha) & \sigma^{3}(\alpha) \\
    \beta & \sigma(\beta) & \sigma^{2}(\beta) & \sigma^{3}(\beta) \\
    \gamma & \sigma(\gamma) & \sigma^{2}(\gamma) & \sigma^{3}(\gamma)
  \end{pmatrix}.
\]

Определим следующие классы трехмерных алгебраических цепных дробей:

\begin{align*}\begin{split}
 \mathbf{CF}_1  = \Big\{ \cf(l_1,l_2,l_3,l_4)\in\gA_3' \,\Big|\, \beta = \sigma(\alpha),\ \gamma = \sigma^{2}(\alpha), \ \trace(\alpha)=0 \Big\}, 
\end{split}\end{align*}
\begin{align*}\begin{split}
 \mathbf{CF}_2  = \Big\{ \cf(l_1,l_2,l_3,l_4)\in\gA_3' \,\Big|\, \beta = \sigma(\alpha) ,\ \gamma = \sigma^{2}(\alpha), \ \trace(\alpha)=1 \Big\},
\end{split}\end{align*}
\begin{align*}\begin{split}
 \mathbf{CF}_3  = \Big\{ \cf(l_1,l_2,l_3,l_4)\in\gA_3' \,\Big|\, \beta = \sigma(\alpha) ,\ \gamma = \sigma^{2}(\alpha), \ \trace(\alpha)=2 \Big\}, 
\end{split}\end{align*}
\begin{align*}\begin{split}
 \mathbf{CF}_4  = \Big\{ \cf(l_1,l_2,l_3,l_4)\in\gA_3' \,\Big|\, \beta = \sigma(\alpha),\ \gamma = \frac{\alpha + \sigma^{2}(\alpha)}{2}, \ \trace(\alpha)=0 \Big\}, 
\end{split}\end{align*}
\begin{align*}\begin{split}
 \mathbf{CF}_5  = \Big\{ \cf(l_1,l_2,l_3,l_4)\in\gA_3' \,\Big|\, \beta = \sigma(\alpha) ,\ \gamma = \frac{\alpha + \sigma^{2}(\alpha)}{2}, \ \trace(\alpha)=2 \Big\}, 
\end{split}\end{align*}
\begin{align*}\begin{split}
 \mathbf{CF}_6  = \Big\{ \cf(l_1,l_2,l_3,l_4)\in\gA_3' \,\Big|\, \beta = \sigma(\alpha) ,\ \gamma = \frac{\alpha + \sigma^{2}(\alpha) + 1}{2}, \ \trace(\alpha)=0 \Big\}, 
\end{split}\end{align*}
\begin{align*}\begin{split}
 \mathbf{CF}_7  = \Big\{ \cf(l_1,l_2,l_3,l_4)\in\gA_3' \,\Big|\, \beta = \sigma(\alpha) ,\ \gamma = \frac{\alpha + \sigma^{2}(\alpha) + 1}{2}, \ \trace(\alpha)=2 \Big\}, 
\end{split}\end{align*}

Покажем, что все дроби из классов $\mathbf{CF}_{i}$, палиндромичны для каждого $i=1,\ldots, 7$. Положим

\[
  G_{1} =
  \begin{pmatrix}
    1 & \phantom{-}0 & \phantom{-}0 & \phantom{-}0 \\
    0 & \phantom{-}0 & \phantom{-}1 & \phantom{-}0\\
    0 & \phantom{-}0 & \phantom{-}0 & \phantom{-}1\\
    0 & -1 & -1 & -1
  \end{pmatrix},
  \quad
    G_{2} =
  \begin{pmatrix}
    1 & \phantom{-}0 &  \phantom{-}0 & \phantom{-}0 \\
    0 & \phantom{-}0 &  \phantom{-}1 & \phantom{-}0 \\
    0 &  \phantom{-}0 &  \phantom{-}0 & \phantom{-}1 \\
    1 & -1 & -1 & -1
  \end{pmatrix},
\]
\[
  G_{3} =
  \begin{pmatrix}
    1 & \phantom{-}0 &  \phantom{-}0 & \phantom{-}0 \\
    0 & \phantom{-}0 &  \phantom{-}1 & \phantom{-}0 \\
    0 &  \phantom{-}0 &  \phantom{-}0 & \phantom{-}1 \\
    2 & -1 & -1 & -1
  \end{pmatrix},
  \quad
    G_{4} =
  \begin{pmatrix}
    1 & \phantom{-}0 & \phantom{-}0 & \phantom{-}0 \\
    0 & \phantom{-}0 & \phantom{-}1 & \phantom{-}0 \\
    0 & -1 & \phantom{-}0 & \phantom{-}2 \\
    0 &\phantom{-}0  & \phantom{-}0 & -1
  \end{pmatrix},
\]
\[
  G_{5} =
  \begin{pmatrix}
    1 & \phantom{-}0 &  \phantom{-}0 & \phantom{-}0 \\
    0 & \phantom{-}0 &  \phantom{-}1 & \phantom{-}0 \\
    0 &  -1 &  \phantom{-}0 & \phantom{-}2 \\
    1 & \phantom{-}0 & \phantom{-}0 & -1
  \end{pmatrix},
  \quad
  G_{6} =
  \begin{pmatrix}
     \phantom{-}1 & \phantom{-}0 &  \phantom{-}0 & \phantom{-}0 \\
     \phantom{-}0 & \phantom{-}0 & \phantom{-}1 & \phantom{-}0 \\
    -1 & -1 & \phantom{-}0 & \phantom{-}2 \\
     \phantom{-}1 & \phantom{-}0 & \phantom{-}0 & -1
  \end{pmatrix},
\]
\[
  G_{7} =
  \begin{pmatrix}
     \phantom{-}1 & \phantom{-}0 &  \phantom{-}0 & \phantom{-}0 \\
     \phantom{-}0 & \phantom{-}0 & \phantom{-}1 & \phantom{-}0 \\
    -1 & -1 & \phantom{-}0 & \phantom{-}2 \\
     \phantom{-}2 & \phantom{-}0 & \phantom{-}0 & -1
  \end{pmatrix}.
\]

\begin{lemma}\label{oper_eq_4d}
  Пусть $\cf(l_1,l_2,l_3,l_4)\in\gA_3$ и $i\in\{1,2,3,4,5,6,7\}$. Тогда цепная дробь $\cf(l_1,l_2,l_3,l_4)$ принадлежит классу $\mathbf{CF}_{i}$ в том и только в том случае, если $G_{i}$ --- её собственная циклическая симметрия.
\end{lemma}

\begin{proof}
  В силу следствия \ref{about_properity_n_4} оператор $G\in\Gl_4(\Z)$ является собственной циклической симметрией дроби $\cf(l_1,l_2,l_3,l_4)$ тогда и только тогда, когда с точностью до перестановки индексов существуют такие действительные числа $\mu_1,\mu_2,\mu_3,\mu_4$, что $G\big(\vec l_1,\vec l_2,\vec l_3,\vec l_4\big)=\big(\mu_2\vec l_2,\mu_3\vec l_3,\mu_4\vec l_4,\mu_1\vec l_1\big)$ и $\mu_{1}\mu_{2}\mu_{3}\mu_{4} = 1$.

  Пусть $\cf(l_1, l_2, l_3, l_4) \in \mathbf{CF}_{1}$. Тогда

\[    G_{1} \big(\vec{l}_1, \vec{l}_2, \vec{l}_3, \vec{l}_4\big) = \]
 \[  \left( \begin{smallmatrix}
      1\phantom{-}\phantom{-} & 1\phantom{-}\phantom{-} & 1\phantom{-}\phantom{-} & 1 \\
      \sigma(\alpha)\phantom{-}\phantom{-} &  \sigma^{2}(\alpha)\phantom{-}\phantom{-} &  \sigma^{3}(\alpha)\phantom{-}\phantom{-} & \alpha \\
      \sigma^{2}(\alpha)\phantom{-}\phantom{-} &  \sigma^{3}(\alpha)\phantom{-}\phantom{-} &  \alpha\phantom{-}\phantom{-} & \sigma(\alpha) \\
      -\alpha - \sigma(\alpha) - \sigma^{2}(\alpha)\phantom{-}\phantom{-} & -\sigma(\alpha) - \sigma^{2}(\alpha) - \sigma^{3}(\alpha)\phantom{-}\phantom{-} & -\sigma^{2}(\alpha) - \sigma^{3}(\alpha) - \alpha\phantom{-}\phantom{-} & -\sigma^{3}(\alpha) - \alpha - \sigma(\alpha) 
    \end{smallmatrix}\right) = \]
\[  \big(\vec{l}_2, \vec{l}_3, \vec{l}_4, \vec{l}_1\big). \]

  Следовательно, $G_{1}$ --- собственная циклическая симметрия $\cf(l_1,l_2,l_3,l_4)$. Обратно, предположим, $G_{1}$ собственная циклическая симметрия $\cf(l_1,l_2,l_3,l_4)$. Тогда существует такое $\mu_2$, что с точностью до перестановки индексов
  \[
    G_{1}\vec{l}_1=
    \begin{pmatrix}
      1  \\
      \beta \\
      \gamma \\
      - \alpha - \beta - \gamma
    \end{pmatrix}
    =\mu_2
    \begin{pmatrix}
      1  \\
      \sigma_{2}(\alpha) \\
      \sigma_{2}(\beta) \\
      \sigma_{2}(\gamma)
    \end{pmatrix}, 
  \]
  откуда $\mu_2 = 1$, $\beta =\sigma_{2}(\alpha)$, $\gamma =\sigma_{2}(\beta)$, $- \alpha - \beta - \gamma=\sigma_{2}(\gamma)$. Существует $\mu_3$, такое что
  \[ 
    G_{1}\vec{l}_2=
    \begin{pmatrix}
      1  \\
      \gamma \\
      - \alpha - \beta - \gamma \\
      \alpha
    \end{pmatrix}
    =\mu_3
    \begin{pmatrix}
      1  \\
      \sigma_{3}(\alpha) \\
      \sigma_{3}(\beta) \\
      \sigma_{3}(\gamma)
    \end{pmatrix},
  \]
   откуда $\mu_3 = 1$, $\sigma_{3}(\alpha) = \gamma = \sigma_{2}(\beta) = \sigma^{2}_{2}(\alpha)$, $\sigma_{3}(\beta) = - \alpha - \beta - \gamma = \sigma_{2}(\gamma) = \sigma^{2}_{2}(\beta)$, $\sigma_{3}(\gamma) = \alpha = -\beta - \gamma - (- \alpha - \beta - \gamma) = \sigma_{2}(- \alpha - \beta - \gamma) = \sigma^{2}_{2}(\gamma)$.
   Существует $\mu_4$, такое что
  \[ 
    G_{1}\vec{l}_3=
    \begin{pmatrix}
      1  \\
      - \alpha - \beta - \gamma  \\
      \alpha \\
      \beta
    \end{pmatrix}
    =\mu_4
    \begin{pmatrix}
      1  \\
      \sigma_{4}(\alpha) \\
      \sigma_{4}(\beta) \\
      \sigma_{4}(\gamma)
    \end{pmatrix},
  \]
   откуда $\mu_4 = 1$, $\sigma_{4}(\alpha) = - \alpha - \beta - \gamma = \sigma_{2}(\gamma) =  \sigma^{3}_{2}(\alpha)$, $\sigma_{4}(\beta) = \alpha = -\beta - \gamma - (- \alpha - \beta - \gamma) = \sigma_{2}(-\alpha - \beta - \gamma) =  \sigma^{3}_{2}(\beta)$, $\sigma_{4}(\gamma) = \beta  = \sigma_{2}(\alpha) = \sigma^{3}_{2}(\gamma)$, $\trace(\alpha) = 0$. Стало быть, $\cf(l_1, l_2, l_3, l_4) \in \mathbf{CF}_{1}$, так как числа $1,\alpha,\beta, \gamma$ образуют базис поля $K=\Q(\alpha,\beta, \gamma)$.

Пусть $\cf(l_1, l_2, l_3, l_4) \in \mathbf{CF}_{2}$. Тогда

\[    G_{2} \big(\vec{l}_1, \vec{l}_2, \vec{l}_3, \vec{l}_4\big) = \]
 \[  \left( \begin{smallmatrix}
      1\phantom{-}\phantom{-} & 1\phantom{-}\phantom{-} & 1\phantom{-}\phantom{-} & 1 \\
      \sigma(\alpha)\phantom{-}\phantom{-} &  \sigma^{2}(\alpha)\phantom{-}\phantom{-} &  \sigma^{3}(\alpha)\phantom{-}\phantom{-} & \alpha \\
      \sigma^{2}(\alpha)\phantom{-}\phantom{-} &  \sigma^{3}(\alpha)\phantom{-}\phantom{-} &  \alpha\phantom{-}\phantom{-} & \sigma(\alpha) \\
      1 -\alpha - \sigma(\alpha) - \sigma^{2}(\alpha)\phantom{-}\phantom{-} & 1 -\sigma(\alpha) - \sigma^{2}(\alpha) - \sigma^{3}(\alpha)\phantom{-}\phantom{-} &1 -\sigma^{2}(\alpha) - \sigma^{3}(\alpha) - \alpha\phantom{-}\phantom{-} &1 -\sigma^{3}(\alpha) - \alpha - \sigma(\alpha) 
    \end{smallmatrix}\right) = \]
\[  \big(\vec{l}_2, \vec{l}_3, \vec{l}_4, \vec{l}_1\big). \]

   Следовательно, $G_{2}$ --- собственная циклическая симметрия $\cf(l_1,l_2,l_3,l_4)$. Обратно, предположим, $G_{2}$ собственная циклическая симметрия $\cf(l_1,l_2,l_3,l_4)$. Тогда существует такое $\mu_2$, что с точностью до перестановки индексов 
   \[
    G_{2}\vec{l}_1=
    \begin{pmatrix}
      1  \\
      \beta \\
      \gamma \\
     1 - \alpha - \beta - \gamma
    \end{pmatrix}
    =\mu_2
    \begin{pmatrix}
      1  \\
      \sigma_{2}(\alpha) \\
      \sigma_{2}(\beta) \\
      \sigma_{2}(\gamma)
    \end{pmatrix}, 
  \]
  откуда $\mu_2 = 1$, $\beta =\sigma_{2}(\alpha)$, $\gamma =\sigma_{2}(\beta)$, $1 - \alpha - \beta - \gamma=\sigma_{2}(\gamma)$. Существует $\mu_3$, такое что
  \[ 
    G_{2}\vec{l}_2=
    \begin{pmatrix}
      1  \\
      \gamma \\
      1 - \alpha - \beta - \gamma \\
      \alpha
    \end{pmatrix}
    =\mu_3
    \begin{pmatrix}
      1  \\
      \sigma_{3}(\alpha) \\
      \sigma_{3}(\beta) \\
      \sigma_{3}(\gamma)
    \end{pmatrix},
  \]
   откуда $\mu_3 = 1$, $\sigma_{3}(\alpha) = \gamma = \sigma_{2}(\beta) = \sigma^{2}_{2}(\alpha)$, $\sigma_{3}(\beta) = 1 - \alpha - \beta - \gamma = \sigma_{2}(\gamma) = \sigma^{2}_{2}(\beta)$, $\sigma_{3}(\gamma) = \alpha = 1 -\beta - \gamma - (1 - \alpha - \beta - \gamma) = \sigma_{2}(1 - \alpha - \beta - \gamma) = \sigma^{2}_{2}(\gamma)$.
   Существует $\mu_4$, такое что
  \[ 
    G_{2}\vec{l}_3=
    \begin{pmatrix}
      1  \\
      1 - \alpha - \beta - \gamma  \\
      \alpha \\
      \beta
    \end{pmatrix}
    =\mu_4
    \begin{pmatrix}
      1  \\
      \sigma_{4}(\alpha) \\
      \sigma_{4}(\beta) \\
      \sigma_{4}(\gamma)
    \end{pmatrix},
  \]
   откуда $\mu_4 = 1$, $\sigma_{4}(\alpha) = 1 - \alpha - \beta - \gamma = \sigma_{2}(\gamma) =  \sigma^{3}_{2}(\alpha)$, $\sigma_{4}(\beta) = \alpha = 1 - \beta - \gamma - (1 - \alpha - \beta - \gamma) = \sigma_{2}(1 - \alpha - \beta - \gamma) =  \sigma^{3}_{2}(\beta)$, $\sigma_{4}(\gamma) = \beta  = \sigma_{2}(\alpha) = \sigma^{3}_{2}(\gamma)$, $\trace(\alpha) = 1$. Стало быть, $\cf(l_1, l_2, l_3, l_4) \in \mathbf{CF}_{2}$, так как числа $1,\alpha,\beta, \gamma$ образуют базис поля $K=\Q(\alpha,\beta, \gamma)$.

Пусть $\cf(l_1, l_2, l_3, l_4) \in \mathbf{CF}_{3}$. Тогда

\[    G_{3} \big(\vec{l}_1, \vec{l}_2, \vec{l}_3, \vec{l}_4\big) = \]
 \[  \left( \begin{smallmatrix}
      1\phantom{-}\phantom{-} & 1\phantom{-}\phantom{-} & 1\phantom{-}\phantom{-} & 1 \\
      \sigma(\alpha)\phantom{-}\phantom{-} &  \sigma^{2}(\alpha)\phantom{-}\phantom{-} &  \sigma^{3}(\alpha)\phantom{-}\phantom{-} & \alpha \\
      \sigma^{2}(\alpha)\phantom{-}\phantom{-} &  \sigma^{3}(\alpha)\phantom{-}\phantom{-} &  \alpha\phantom{-}\phantom{-} & \sigma(\alpha) \\
      2 -\alpha - \sigma(\alpha) - \sigma^{2}(\alpha)\phantom{-}\phantom{-} & 2 -\sigma(\alpha) - \sigma^{2}(\alpha) - \sigma^{3}(\alpha)\phantom{-}\phantom{-} & 2 -\sigma^{2}(\alpha) - \sigma^{3}(\alpha) - \alpha\phantom{-}\phantom{-} & 2 - \sigma^{3}(\alpha) - \alpha - \sigma(\alpha) 
    \end{smallmatrix}\right) = \]
\[  \big(\vec{l}_2, \vec{l}_3, \vec{l}_4, \vec{l}_1\big). \]

   Следовательно, $G_{3}$ --- собственная циклическая симметрия $\cf(l_1,l_2,l_3,l_4)$. Обратно, предположим, $G_{3}$ собственная циклическая симметрия $\cf(l_1,l_2,l_3,l_4)$. Тогда существует такое $\mu_2$, что с точностью до перестановки индексов
  \[
    G_{3}\vec{l}_1=
    \begin{pmatrix}
      1  \\
      \beta \\
      \gamma \\
      2 - \alpha - \beta - \gamma
    \end{pmatrix}
    =\mu_2
    \begin{pmatrix}
      1  \\
      \sigma_{2}(\alpha) \\
      \sigma_{2}(\beta) \\
      \sigma_{2}(\gamma)
    \end{pmatrix}, 
  \]
  откуда $\mu_2 = 1$, $\beta =\sigma_{2}(\alpha)$, $\gamma =\sigma_{2}(\beta)$, $2 - \alpha - \beta - \gamma =\sigma_{2}(\gamma)$. Существует $\mu_3$, такое что   
  \[ 
    G_{3}\vec{l}_2=
    \begin{pmatrix}
      1  \\
      \gamma \\
      2 - \alpha - \beta - \gamma \\
      \alpha
    \end{pmatrix}
    =\mu_3
    \begin{pmatrix}
      1  \\
      \sigma_{3}(\alpha) \\
      \sigma_{3}(\beta) \\
      \sigma_{3}(\gamma)
    \end{pmatrix},
  \]
   откуда $\mu_3 = 1$, $\sigma_{3}(\alpha) = \gamma = \sigma_{2}(\beta) = \sigma^{2}_{2}(\alpha)$, $\sigma_{3}(\beta) = 2 - \alpha - \beta - \gamma = \sigma_{2}(\gamma) = \sigma^{2}_{2}(\beta)$, $\sigma_{3}(\gamma) = \alpha = 2 - \beta - \gamma - (2 -\alpha - \beta -\gamma) = \sigma_{2}(2 - \alpha - \beta - \gamma) = \sigma^{2}_{2}(\gamma)$. Существует $\mu_4$, такое что
  \[ 
    G_{3}\vec{l}_3=
    \begin{pmatrix}
      1  \\
      2 - \alpha - \beta - \gamma  \\
      \alpha \\
      \beta
    \end{pmatrix}
    =\mu_4
    \begin{pmatrix}
      1  \\
      \sigma_{4}(\alpha) \\
      \sigma_{4}(\beta) \\
      \sigma_{4}(\gamma)
    \end{pmatrix},
  \]
   откуда $\mu_4 = 1$, $\sigma_{4}(\alpha) = 2 - \alpha - \beta - \gamma = \sigma_{2}(\gamma) = \sigma^{3}_{2}(\alpha)$, $\sigma_{4}(\beta) = \alpha = 2 - \beta - \gamma - (2 -\alpha - \beta -\gamma) = \sigma_{2}(2 - \alpha - \beta - \gamma) = \sigma^{3}_{2}(\beta)$, $\sigma_{4}(\gamma) = \beta = \sigma_{2}(\alpha) = \sigma^{3}_{2}(\gamma)$, $\trace(\alpha) = 2$. Стало быть, $\cf(l_1, l_2, l_3, l_4) \in \mathbf{CF}_{3}$, так как числа $1,\alpha,\beta, \gamma$ образуют базис поля $K=\Q(\alpha,\beta, \gamma)$.

   Пусть $\cf(l_1, l_2, l_3, l_4) \in \mathbf{CF}_{4}$. Тогда

\[    G_{4} \big(\vec{l}_1, \vec{l}_2, \vec{l}_3, \vec{l}_4\big) = \]
 \[  \left( \begin{smallmatrix}
      1\phantom{-}\phantom{-} & 1\phantom{-}\phantom{-} & 1\phantom{-}\phantom{-} & 1 \\
      \sigma(\alpha)\phantom{-}\phantom{-} &  \sigma^{2}(\alpha)\phantom{-}\phantom{-} &  \sigma^{3}(\alpha)\phantom{-}\phantom{-} & \alpha \\
      \sigma^{2}(\alpha)\phantom{-}\phantom{-} &  \sigma^{3}(\alpha)\phantom{-}\phantom{-} &  \alpha\phantom{-}\phantom{-} & \sigma(\alpha) \\
      -\frac{\alpha + \sigma^{2}(\alpha)}{2}\phantom{-}\phantom{-} &  -\frac{\sigma(\alpha) + \sigma^{3}(\alpha)}{2}\phantom{-}\phantom{-} &  -\frac{\sigma^{2}(\alpha) + \alpha}{2}\phantom{-}\phantom{-} & -\frac{\sigma^{3}(\alpha) + \sigma(\alpha)}{2}
    \end{smallmatrix}\right) = \]
\[  \big(\vec{l}_2, \vec{l}_3, \vec{l}_4, \vec{l}_1\big). \]

 Следовательно, $G_{4}$ --- собственная циклическая симметрия $\cf(l_1,l_2,l_3,l_4)$. Обратно, предположим, $G_{4}$ собственная циклическая симметрия $\cf(l_1,l_2,l_3,l_4)$. Тогда существует такое $\mu_2$, что с точностью до перестановки индексов
  \[
    G_{4}\vec{l}_1=
    \begin{pmatrix}
      1  \\
      \beta \\
      - \alpha + 2\gamma\\
      - \gamma
    \end{pmatrix}
    =\mu_2
    \begin{pmatrix}
      1  \\
      \sigma_{2}(\alpha) \\
      \sigma_{2}(\beta) \\
      \sigma_{2}(\gamma)
    \end{pmatrix}, 
  \]
откуда $\mu_2 = 1$, $\beta =\sigma_{2}(\alpha)$, $- \alpha + 2\gamma =\sigma_{2}(\beta)$, $ - \gamma =\sigma_{2}(\gamma)$. Существует $\mu_3$, такое что
    \[ 
    G_{4}\vec{l}_2=
    \begin{pmatrix}
      1  \\
      - \alpha + 2\gamma  \\
      - \beta - 2\gamma \\
      \gamma
    \end{pmatrix}
    =\mu_3
    \begin{pmatrix}
      1  \\
      \sigma_{3}(\alpha) \\
      \sigma_{3}(\beta) \\
      \sigma_{3}(\gamma)
    \end{pmatrix},
  \]
откуда $\mu_3 = 1$, $\frac{\alpha + \sigma_{3}(\alpha)}{2} = \gamma, \sigma_{3}(\alpha) = - \alpha + 2\gamma = \sigma_{2}(\beta) = \sigma^{2}_{2}(\alpha)$,  $\sigma_{3}(\beta) = - \beta - 2\gamma = \sigma_{2}(-\alpha + 2\gamma) = \sigma^{2}_{2}(\beta)$, $\sigma_{3}(\gamma) = \gamma = \sigma_{2}(-\gamma) = \sigma^{2}_{2}(\gamma)$. Существует $\mu_4$, такое что
 \[ 
    G_{4}\vec{l}_3=
    \begin{pmatrix}
      1  \\
      - \beta -  2\gamma  \\
      \alpha \\
      -\gamma
    \end{pmatrix}
    =\mu_4
    \begin{pmatrix}
      1  \\
      \sigma_{4}(\alpha) \\
      \sigma_{4}(\beta) \\
      \sigma_{4}(\gamma)
    \end{pmatrix},
  \]
откуда $\mu_4 = 1$, $\sigma_{4}(\alpha) = - \beta -  2\gamma = \sigma_{2}(-\alpha) + \sigma_{2}(2\gamma) = \sigma_{2}(-\alpha + 2\gamma) = \sigma^{3}_{2}(\alpha)$, $\sigma_{4}(\beta) = \alpha = 2\gamma - \sigma_{2}(\beta)  = \sigma_{2}(-2\gamma - \beta) =  \sigma^{3}_{2}(\beta)$, $\sigma_{4}(\gamma) = -\gamma = \sigma_{2}(\gamma)  =  \sigma^{3}_{2}(\gamma)$, $\trace(\alpha) = 0$. Стало быть, $\cf(l_1, l_2, l_3, l_4) \in \mathbf{CF}_{4}$, так как числа $1,\alpha,\beta, \gamma$ образуют базис поля $K=\Q(\alpha,\beta, \gamma)$.

Пусть $\cf(l_1, l_2, l_3, l_4) \in \mathbf{CF}_{5}$. Тогда

\[    G_{5} \big(\vec{l}_1, \vec{l}_2, \vec{l}_3, \vec{l}_4\big) = \]
 \[  \left( \begin{smallmatrix}
      1\phantom{-}\phantom{-} & 1\phantom{-}\phantom{-} & 1\phantom{-}\phantom{-} & 1 \\
      \sigma(\alpha)\phantom{-}\phantom{-} &  \sigma^{2}(\alpha)\phantom{-}\phantom{-} &  \sigma^{3}(\alpha)\phantom{-}\phantom{-} & \alpha \\
      \sigma^{2}(\alpha)\phantom{-}\phantom{-} &  \sigma^{3}(\alpha)\phantom{-}\phantom{-} &  \alpha\phantom{-}\phantom{-} & \sigma(\alpha) \\
      \frac{2 - \alpha - \sigma^{2}(\alpha)}{2}\phantom{-}\phantom{-} &  \frac{2 - \sigma(\alpha) - \sigma^{3}(\alpha)}{2}\phantom{-}\phantom{-} &  \frac{2 - \sigma^{2}(\alpha) - \alpha}{2}\phantom{-}\phantom{-} & \frac{2 - \sigma^{3}(\alpha) - \sigma(\alpha)}{2}
    \end{smallmatrix}\right) = \]
\[  \big(\vec{l}_2, \vec{l}_3, \vec{l}_4, \vec{l}_1\big). \]

   Следовательно, $G_{5}$ --- собственная циклическая симметрия $\cf(l_1,l_2,l_3,l_4)$. Обратно, предположим, $G_{5}$ собственная циклическая симметрия $\cf(l_1,l_2,l_3,l_4)$. Тогда существует такое $\mu_2$, что с точностью до перестановки индексов
   \[
    G_{5}\vec{l}_1=
    \begin{pmatrix}
      1  \\
      \beta \\
      - \alpha + 2\gamma\\
      1 - \gamma
    \end{pmatrix}
    =\mu_2
    \begin{pmatrix}
      1  \\
      \sigma_{2}(\alpha) \\
      \sigma_{2}(\beta) \\
      \sigma_{2}(\gamma)
    \end{pmatrix}, 
  \]
откуда $\mu_2 = 1$, $\beta =\sigma_{2}(\alpha)$, $- \alpha + 2\gamma =\sigma_{2}(\beta)$, $ 1 - \gamma =\sigma_{2}(\gamma)$. Существует $\mu_3$, такое что
    \[ 
    G_{5}\vec{l}_2=
    \begin{pmatrix}
      1  \\
      - \alpha + 2\gamma  \\
     2 - \beta - 2\gamma \\
      \gamma
    \end{pmatrix}
    =\mu_3
    \begin{pmatrix}
      1  \\
      \sigma_{3}(\alpha) \\
      \sigma_{3}(\beta) \\
      \sigma_{3}(\gamma)
    \end{pmatrix},
  \]
откуда $\mu_3 = 1$, $\frac{\alpha + \sigma_{3}(\alpha)}{2} = \gamma, \sigma_{3}(\alpha) = - \alpha + 2\gamma = \sigma_{2}(\beta) = \sigma^{2}_{2}(\alpha)$,  $\sigma_{3}(\beta) = 2 - \beta - 2\gamma = \sigma_{2}(-\alpha + 2\gamma) = \sigma^{2}_{2}(\beta)$, $\sigma_{3}(\gamma) = \gamma = \sigma_{2}(1 - \gamma) = \sigma^{2}_{2}(\gamma)$. Существует $\mu_4$, такое что
 \[ 
    G_{5}\vec{l}_3=
    \begin{pmatrix}
      1  \\
      2 - \beta -  2\gamma  \\
      \alpha \\
      1 -\gamma
    \end{pmatrix}
    =\mu_4
    \begin{pmatrix}
      1  \\
      \sigma_{4}(\alpha) \\
      \sigma_{4}(\beta) \\
      \sigma_{4}(\gamma)
    \end{pmatrix},
  \]
откуда $\mu_4 = 1$, $\sigma_{4}(\alpha) = 2 - \beta -  2\gamma = \sigma_{2}(-\alpha) + \sigma_{2}(2\gamma) = \sigma_{2}(-\alpha + 2\gamma) = \sigma^{3}_{2}(\alpha)$, $\sigma_{4}(\beta) = \alpha = 2\gamma - \sigma_{2}(\beta)  = \sigma_{2}(2 - 2\gamma - \beta) =  \sigma^{3}_{2}(\beta)$, $\sigma_{4}(\gamma) = 1 -\gamma = \sigma_{2}(\gamma)  =  \sigma^{3}_{2}(\gamma)$, $\trace(\alpha) = 2$. Стало быть, $\cf(l_1, l_2, l_3, l_4) \in \mathbf{CF}_{5}$, так как числа $1,\alpha,\beta, \gamma$ образуют базис поля $K=\Q(\alpha,\beta, \gamma)$.

Пусть $\cf(l_1, l_2, l_3, l_4) \in \mathbf{CF}_{6}$. Тогда

\[    G_{6} \big(\vec{l}_1, \vec{l}_2, \vec{l}_3, \vec{l}_4\big) = \]
 \[  \left( \begin{smallmatrix}
      1\phantom{-}\phantom{-} & 1\phantom{-}\phantom{-} & 1\phantom{-}\phantom{-} & 1 \\
      \sigma(\alpha)\phantom{-}\phantom{-} &  \sigma^{2}(\alpha)\phantom{-}\phantom{-} &  \sigma^{3}(\alpha)\phantom{-}\phantom{-} & \alpha \\
      \sigma^{2}(\alpha)\phantom{-}\phantom{-} &  \sigma^{3}(\alpha)\phantom{-}\phantom{-} &  \alpha\phantom{-}\phantom{-} & \sigma(\alpha) \\
      \frac{1 - \alpha - \sigma^{2}(\alpha)}{2}\phantom{-}\phantom{-} &  \frac{1 - \sigma(\alpha) - \sigma^{3}(\alpha)}{2}\phantom{-}\phantom{-} &  \frac{1 - \sigma^{2}(\alpha) - \alpha}{2}\phantom{-}\phantom{-} & \frac{1 - \sigma^{3}(\alpha) - \sigma(\alpha)}{2}
    \end{smallmatrix}\right) = \]
\[  \big(\vec{l}_2, \vec{l}_3, \vec{l}_4, \vec{l}_1\big). \]

 Следовательно, $G_{6}$ --- собственная циклическая симметрия $\cf(l_1,l_2,l_3,l_4)$. Обратно, предположим, $G_{6}$ собственная циклическая симметрия $\cf(l_1,l_2,l_3,l_4)$. Тогда существует такое $\mu_2$, что с точностью до перестановки индексов
  \[
    G_{6}\vec{l}_1=
    \begin{pmatrix}
      1  \\
      \beta \\
      -1 - \alpha + 2\gamma \\
      1 - \gamma
    \end{pmatrix}
    =\mu_2
    \begin{pmatrix}
      1  \\
      \sigma_{2}(\alpha) \\
      \sigma_{2}(\beta) \\
      \sigma_{2}(\gamma)
    \end{pmatrix}, 
  \]
откуда $\mu_2 = 1$, $\beta =\sigma_{2}(\alpha)$, $-1 - \alpha + 2\gamma  =\sigma_{2}(\beta)$, $1 - \gamma =\sigma_{2}(\gamma)$. Существует $\mu_3$, такое что
   \[ 
    G_{6}\vec{l}_2=
    \begin{pmatrix}
      1  \\
      -1 - \alpha + 2\gamma  \\
      1 - \beta - 2\gamma\\
      \gamma
    \end{pmatrix}
    =\mu_3
    \begin{pmatrix}
      1  \\
      \sigma_{3}(\alpha) \\
      \sigma_{3}(\beta) \\
      \sigma_{3}(\gamma)
    \end{pmatrix},
  \]
откуда $\mu_3 = 1$, $\frac{\alpha + \sigma_{3}(\alpha) + 1}{2} = \gamma, \sigma_{3}(\alpha) = -1 - \alpha + 2\gamma = \sigma_{2}(\beta) = \sigma^{2}_{2}(\alpha)$, $\sigma_{3}(\beta) = 1 - \beta - 2\gamma = -1 -\beta + (2 - 2\gamma) = \sigma_{2}(-1 -\alpha + 2\gamma) = \sigma^{2}_{2}(\beta)$, $\sigma_{3}(\gamma) = \gamma = \sigma_{2}(1 - \gamma) = \sigma^{3}_{2}(\gamma)$. Существует $\mu_4$, такое что
\[ 
    G_{6}\vec{l}_3=
    \begin{pmatrix}
      1  \\
      1 - \beta -  2\gamma \\
      \alpha \\
      1 - \gamma
    \end{pmatrix}
    =\mu_4
    \begin{pmatrix}
      1  \\
      \sigma_{4}(\alpha) \\
      \sigma_{4}(\beta) \\
      \sigma_{4}(\gamma)
    \end{pmatrix},
  \]
откуда $\mu_4 = 1$, $\sigma_{4}(\alpha) = 1 - \beta -  2\gamma = -1 -\beta + (2 - 2\gamma) = \sigma_{2}(-1 -\alpha + 2\gamma) = \sigma^{3}_{2}(\alpha)$, $\sigma_{4}(\beta) = \alpha = 1 -(-1 -\alpha + 2\gamma) - (2 - 2\gamma) = \sigma_{2}(1 -\beta - 2\gamma) = \sigma^{3}_{2}(\beta)$, $\sigma_{4}(\gamma) = 1 - \gamma = \sigma_{2}(\gamma) = \sigma^{3}_{2}(\gamma)$, $\trace(\alpha) = 0$. Стало быть, $\cf(l_1, l_2, l_3, l_4) \in \mathbf{CF}_{6}$, так как числа $1,\alpha,\beta, \gamma$ образуют базис поля $K=\Q(\alpha,\beta, \gamma)$.

 Пусть $\cf(l_1, l_2, l_3, l_4) \in \mathbf{CF}_{7}$. Тогда

\[    G_{7} \big(\vec{l}_1, \vec{l}_2, \vec{l}_3, \vec{l}_4\big) = \]
 \[  \left( \begin{smallmatrix}
      1\phantom{-}\phantom{-} & 1\phantom{-}\phantom{-} & 1\phantom{-}\phantom{-} & 1 \\
      \sigma(\alpha)\phantom{-}\phantom{-} &  \sigma^{2}(\alpha)\phantom{-}\phantom{-} &  \sigma^{3}(\alpha)\phantom{-}\phantom{-} & \alpha \\
      \sigma^{2}(\alpha)\phantom{-}\phantom{-} &  \sigma^{3}(\alpha)\phantom{-}\phantom{-} &  \alpha\phantom{-}\phantom{-} & \sigma(\alpha) \\
      \frac{3 - \alpha - \sigma^{2}(\alpha)}{2}\phantom{-}\phantom{-} &  \frac{3 - \sigma(\alpha) - \sigma^{3}(\alpha)}{2}\phantom{-}\phantom{-} &  \frac{3 - \sigma^{2}(\alpha) - \alpha}{2}\phantom{-}\phantom{-} & \frac{3 - \sigma^{3}(\alpha) - \sigma(\alpha)}{2}
    \end{smallmatrix}\right) = \]
\[  \big(\vec{l}_2, \vec{l}_3, \vec{l}_4, \vec{l}_1\big). \]

 Следовательно, $G_{7}$ --- собственная циклическая симметрия $\cf(l_1,l_2,l_3,l_4)$. Обратно, предположим, $G_{7}$ собственная циклическая симметрия $\cf(l_1,l_2,l_3,l_4)$. Тогда существует такое $\mu_2$, что с точностью до перестановки индексов
  \[
    G_{7}\vec{l}_1=
    \begin{pmatrix}
      1  \\
      \beta \\
      -1 - \alpha + 2\gamma \\
      2 - \gamma
    \end{pmatrix}
    =\mu_2
    \begin{pmatrix}
      1  \\
      \sigma_{2}(\alpha) \\
      \sigma_{2}(\beta) \\
      \sigma_{2}(\gamma)
    \end{pmatrix}, 
  \]
откуда $\mu_2 = 1$, $\beta =\sigma_{2}(\alpha)$, $-1 - \alpha + 2\gamma = \sigma_{2}(\beta)$, $2 - \gamma =\sigma_{2}(\gamma)$. Существует $\mu_3$, такое что
   \[ 
    G_{7}\vec{l}_2=
    \begin{pmatrix}
      1  \\
      -1 - \alpha + 2\gamma  \\
      3 - \beta - 2\gamma\\
      \gamma
    \end{pmatrix}
    =\mu_3
    \begin{pmatrix}
      1  \\
      \sigma_{3}(\alpha) \\
      \sigma_{3}(\beta) \\
      \sigma_{3}(\gamma)
    \end{pmatrix},
  \]
откуда $\mu_3 = 1$, $\frac{\alpha + \sigma_{3}(\alpha) + 1}{2} = \gamma, \sigma_{3}(\alpha) = -1 - \alpha + 2\gamma = \sigma_{2}(\beta) = \sigma^{2}_{2}(\alpha)$, $\sigma_{3}(\beta) = 3 - \beta - 2\gamma = -1 -\beta + (4 - 2\gamma) = \sigma_{2}(-1 -\alpha + 2\gamma) = \sigma^{2}_{2}(\beta)$, $\sigma_{3}(\gamma) = \gamma = \sigma_{2}(2 - \gamma) = \sigma^{3}_{2}(\gamma)$. Существует $\mu_4$, такое что
\[ 
    G_{7}\vec{l}_3=
    \begin{pmatrix}
      1  \\
      3 - \beta -  2\gamma \\
      \alpha \\
      2- \gamma
    \end{pmatrix}
    =\mu_4
    \begin{pmatrix}
      1  \\
      \sigma_{4}(\alpha) \\
      \sigma_{4}(\beta) \\
      \sigma_{4}(\gamma)
    \end{pmatrix},
  \]
откуда $\mu_4 = 1$, $\sigma_{4}(\alpha) = 3 - \beta -  2\gamma = -1 -\beta + (4 - 2\gamma) = \sigma_{2}(-1 -\alpha + 2\gamma) = \sigma^{3}_{2}(\alpha)$, $\sigma_{4}(\beta) = \alpha = 3 -(-\alpha + 2\gamma - 1) - (4 - 2\gamma) = \sigma_{2}(3 -\beta - 2\gamma) = \sigma^{3}_{2}(\beta)$, $\sigma_{4}(\gamma) = 2 - \gamma = \sigma_{2}(\gamma) = \sigma^{3}_{2}(\gamma)$, $\trace(\alpha) = 2$. Стало быть, $\cf(l_1, l_2, l_3, l_4) \in \mathbf{CF}_{7}$, так как числа $1,\alpha,\beta, \gamma$ образуют базис поля $K=\Q(\alpha,\beta, \gamma)$.
\end{proof}

Обозначим также для каждого $i=1,\ldots,7$ через $\overline{\mathbf{CF}}_i$ образ $\mathbf{CF}_i$ при действии группы $\Gl_4(\Z)$:
\[
  \overline{\mathbf{CF}}_i=
  \Big\{ \cf(l_1,l_2,l_3,l_4)\in\gA_3' \,\Big|\, \exists X\in\Gl_4(\Z):X\big(\cf(l_1,l_2,l_3,l_4)\big)\in\mathbf{CF}_i \Big\}.
\]

\begin{lemma}\label{l:CF_instead_of_statements}
  Для дроби $\cf(l_1,l_2,l_3,l_4)\in\gA_3$ выполняется условие $(i)$ теоремы \ref{main_t_4} тогда и только тогда, когда $\cf(l_1,l_2,l_3,l_4)$ принадлежит классу $\overline{\mathbf{CF}}_i$, где $i\in\{1,2,3,4,5,6,7\}$.
\end{lemma}

\begin{proof}
  Для любого $X\in\Gl_4(\Z)$ гиперболичность оператора $A\in\Gl_4(\Z)$ равносильна гиперболичности оператора $XAX^{-1}$. При этом собственные подпространства гиперболического оператора однозначно восстанавливаются по любому его собственному вектору. Остаётся воспользоваться определением эквивалентности из параграфа \ref{intro}.
\end{proof}

Теорему \ref{main_t_4} при помощи леммы \ref{l:CF_instead_of_statements}  можно переформулировать следующим образом: \emph{дробь $\cf(l_1,l_2,l_3,l_4)\in\gA_3$ имеет собственную циклическую симметрию $G$ тогда и только тогда, когда $\cf(l_1,l_2,l_3,l_4)$ принадлежит одному из классов $\overline{\mathbf{CF}}_i$, где $i\in\{1,2,3,4,5,6,7\}$}.

\begin{proof}[Доказательство теоремы \ref{main_t_4}]
Если $\cf(l_1,l_2,l_3,l_4)$ принадлежит какому-то $\overline{\mathbf{CF}}_i$, то по лемме \ref{oper_eq_4d} она имеет собственную циклическую симметрию $G$, ибо действие оператора из $\Gl_4(\Z)$ сохраняет свойство существования у алгебраической цепной дроби собственной циклической симметрии.

Обратно, пусть дробь $\cf(l_1,l_2,l_3,l_4)\in\gA_3$ имеет собственную циклическую симметрию $G$ с неподвижной точкой на некотором парусе $\partial(\cK(C)) \in \cf(l_1,l_2, l_3, l_4)$. Рассмотрим точки $\vec z_1$, $\vec z_2$, $\vec z_3$, $\vec z_4$ из леммы \ref{main_lem}. Обозначим также через $\vec{e}_1$, $\vec{e}_2$, $\vec{e}_3$, $\vec{e}_4$ стандартный базис $\R^4$. Для точек $\vec z_1$, $\vec z_2$, $\vec z_3$, $\vec z_4$ выполняется хотя бы одно из утверждений \textup{(1)} - \textup{(7)} леммы \ref{main_lem}.

Пусть выполняется утверждение \textup{(1)} леммы \ref{main_lem}. Рассмотрим такой оператор $X_{1} \in \Gl_4(\Z)$, что
\[X_{1}\big(\vec{z}_1, \vec{z}_2, \vec{z}_3, \frac{1}{4}(\vec{z}_{1}+\vec{z}_{2}+\vec{z}_{3}+\vec{z}_{4})\big) = \big(\vec{e}_1 - \vec{e}_2, \vec{e}_1 + \vec{e}_4, \vec{e}_1 + \vec{e}_3 - \vec{e}_4, \vec{e}_1\big).\]
Тогда $X_{1}(\vec{z}_4) = X_{1}(4\cdot\frac{1}{4}(\vec{z}_{1}+\vec{z}_{2}+\vec{z}_{3}+\vec{z}_{4}) - \vec{z}_1 - \vec{z}_2 - \vec{z}_3) = \vec{e}_1 + \vec{e}_2 - \vec{e}_3$ и $X_{1}GX_{1}^{-1} = G_{1}$, так как по лемме \ref{main_lem}
\[X_{1}GX_{1}^{-1}\big(\vec{e}_1 - \vec{e}_2, \vec{e}_1 + \vec{e}_4, \vec{e}_1 + \vec{e}_3 - \vec{e}_4, \vec{e}_1 + \vec{e}_2 - \vec{e}_3\big) = \]
\[\big(\vec{e}_1 + \vec{e}_4, \vec{e}_1 + \vec{e}_3 - \vec{e}_4, \vec{e}_1 + \vec{e}_2 - \vec{e}_3, \vec{e}_1 - \vec{e}_2\big).\]
Стало быть, $X_{1}\big(\cf(l_1,l_2,l_3,l_4)\big) \in  \mathbf{CF}_{1}$, то есть $\cf(l_1,l_2,l_3,l_4)\in\overline{\mathbf{CF}}_1$.

Пусть выполняется утверждение \textup{(2)} леммы \ref{main_lem}. Рассмотрим такой оператор $X_{2} \in \Gl_4(\Z)$, что
\[X_{2}\big(\vec{z}_1, \vec{z}_2, \vec{z}_3, \vec{z}_{4}\big) = \big(\vec{e}_1, \vec{e}_1 + \vec{e}_4, \vec{e}_1 + \vec{e}_3, \vec{e}_1 + \vec{e}_2\big).\]
Тогда $X_{2}GX_{2}^{-1} = G_{2}$, так как по лемме \ref{main_lem}
\[X_{2}GX_{2}^{-1}\big(\vec{e}_1, \vec{e}_1 + \vec{e}_4, \vec{e}_1 + \vec{e}_3, \vec{e}_1 + \vec{e}_2\big) = \]
\[\big( \vec{e}_1 + \vec{e}_4, \vec{e}_1 + \vec{e}_3, \vec{e}_1 + \vec{e}_2, \vec{e}_1\big).\]
Стало быть, $X_{2}\big(\cf(l_1,l_2,l_3,l_4)\big) \in  \mathbf{CF}_{2}$, то есть $\cf(l_1,l_2,l_3,l_4)\in\overline{\mathbf{CF}}_2$.

Пусть выполняется утверждение \textup{(3)} леммы \ref{main_lem}. Рассмотрим такой оператор $X_{3} \in \Gl_4(\Z)$, что
\[X_{3}\big(\vec{z}_1, \frac{1}{2}(\vec{z}_{1}+\vec{z}_{2}), \frac{1}{2}(\vec{z}_{1}+\vec{z}_{3}), \frac{1}{2}(\vec{z}_{1}+\vec{z}_{4})\big) = \big(\vec{e}_1, \vec{e}_1 + \vec{e}_4, \vec{e}_1 + \vec{e}_3, \vec{e}_1 + \vec{e}_2 \big).\]
Тогда $X_{3}(\vec{z}_2) = X_{3}(2\cdot\frac{1}{2}(\vec{z}_1 + \vec{z}_2) - \vec{z}_1) = \vec{e}_1  + 2\vec{e}_4$, $X_{3}(\vec{z}_3) = X_{3}(2\cdot\frac{1}{2}(\vec{z}_1 + \vec{z}_3) - \vec{z}_1) = \vec{e}_1  + 2\vec{e}_3$, $X_{3}(\vec{z}_4) = X_{3}(2\cdot\frac{1}{2}(\vec{z}_{1}+\vec{z}_{4})-\vec{z}_1) = \vec{e}_1 + 2\vec{e}_2$ и $X_{3}GX_{3}^{-1} = G_{3}$, так как по лемме \ref{main_lem}
\[X_{3}GX_{3}^{-1}\big(\vec{e}_1, \vec{e}_1  + 2\vec{e}_4, \vec{e}_1 + 2\vec{e}_3, \vec{e}_1 + 2\vec{e}_2\big) = \]
\[\big(\vec{e}_1  + 2\vec{e}_4, \vec{e}_1 + 2\vec{e}_3, \vec{e}_1 + 2\vec{e}_2, \vec{e}_1\big).\]
Стало быть, $X_{3}\big(\cf(l_1,l_2,l_3,l_4)\big) \in  \mathbf{CF}_{3}$, то есть $\cf(l_1,l_2,l_3,l_4)\in\overline{\mathbf{CF}}_3$.

Пусть выполняется утверждение \textup{(4)} леммы \ref{main_lem}. Рассмотрим такой оператор $X_{4} \in \Gl_4(\Z)$, что
\[X_{4}\big(\vec{z}_1, \vec{z}_2, \frac{1}{2}(\vec{z}_1 + \vec{z}_3),  \frac{1}{4}(\vec{z}_{1}+\vec{z}_{2}+\vec{z}_{3}+\vec{z}_{4})\big) = \big(\vec{e}_1 - \vec{e}_3 + \vec{e}_4, \vec{e}_1 -\vec{e}_2 + 2\vec{e}_3 - \vec{e}_4, \vec{e}_1 + \vec{e}_2 - \vec{e}_3 +\vec{e}_4, \vec{e}_1\big).\]
Тогда $X_{4}(\vec{z}_3) = X_{4}(2\cdot\frac{1}{2}(\vec{z}_1 + \vec{z}_3) - \vec{z}_1) = \vec{e}_1 + 2\vec{e}_2 - \vec{e}_3 + \vec{e}_4$, $X_{4}(\vec{z}_4) = X_{4}(4\cdot\frac{1}{4}(\vec{z}_{1}+\vec{z}_{2}+\vec{z}_{3}+\vec{z}_{4}) - \vec{z}_1 - \vec{z}_2 - \vec{z}_3) = \vec{e}_1 - \vec{e}_2 - \vec{e}_4$ и $X_{4}GX_{4}^{-1} = G_{4}$, так как по лемме \ref{main_lem}
\[X_{4}GX_{4}^{-1}\big(\vec{e}_1 - \vec{e}_3 + \vec{e}_4, \vec{e}_1 -\vec{e}_2 + 2\vec{e}_3 - \vec{e}_4, \vec{e}_1 + 2\vec{e}_2 - \vec{e}_3 + \vec{e}_4, \vec{e}_1 - \vec{e}_2 - \vec{e}_4\big) = \]
\[\big(\vec{e}_1 -\vec{e}_2 + 2\vec{e}_3 - \vec{e}_4, \vec{e}_1 + 2\vec{e}_2 - \vec{e}_3 + \vec{e}_4, \vec{e}_1 - \vec{e}_2 - \vec{e}_4, \vec{e}_1 - \vec{e}_3 + \vec{e}_4\big).\]
Стало быть, $X_{4}\big(\cf(l_1,l_2,l_3,l_4)\big) \in  \mathbf{CF}_{4}$, то есть $\cf(l_1,l_2,l_3,l_4)\in\overline{\mathbf{CF}}_4$.

Пусть выполняется утверждение \textup{(5)} леммы \ref{main_lem}. Рассмотрим такой оператор $X_{5} \in \Gl_4(\Z)$, что
\[X_{5}\big(\vec{z}_1, \vec{z}_2, \frac{1}{2}(\vec{z}_1 + \vec{z}_3),  \frac{1}{2}(\vec{z}_2 + \vec{z}_4)\big) = \big(\vec{e}_1, \vec{e}_1 + \vec{e}_4, \vec{e}_1 + \vec{e}_3, \vec{e}_1 + \vec{e}_2\big).\]
Тогда $X_{5}(\vec{z}_3) = X_{5}(2\cdot\frac{1}{2}(\vec{z}_1 + \vec{z}_3) - \vec{z}_1) = \vec{e}_1 + 2\vec{e}_3$, $X_{5}(\vec{z}_4) = X_{5}(2\cdot\frac{1}{2}(\vec{z}_2 + \vec{z}_4) - \vec{z}_2) = \vec{e}_1 + 2\vec{e}_2 + \vec{e}_4$ и $X_{5}GX_{5}^{-1} = G_{5}$, так как по лемме \ref{main_lem}
\[X_{5}GX_{5}^{-1}\big(\vec{e}_1, \vec{e}_1 + \vec{e}_4, \vec{e}_1 + 2\vec{e}_3, \vec{e}_1 + 2\vec{e}_2 + \vec{e}_4\big) = \]
\[\big( \vec{e}_1 + \vec{e}_4, \vec{e}_1 + 2\vec{e}_3, \vec{e}_1 + 2\vec{e}_2 + \vec{e}_4, \vec{e}_1\big).\]
Стало быть, $X_{5}\big(\cf(l_1,l_2,l_3,l_4)\big) \in  \mathbf{CF}_{5}$, то есть $\cf(l_1,l_2,l_3,l_4)\in\overline{\mathbf{CF}}_5$.

Пусть выполняется утверждение \textup{(6)} леммы \ref{main_lem}. Рассмотрим такой оператор $X_{6} \in \Gl_4(\Z)$, что
\[X_{6}\big(\vec{z}_1, \vec{z}_{2}, \vec{z}_{3}, \frac{1}{2}(\vec{z}_{1} + \vec{z}_{3} + \vec{z}_{4} - \vec{z}_{2})\big) =\]
\[\big(\vec{e}_1 + \vec{e}_2 + \vec{e}_4, \vec{e}_1, \vec{e}_1 - \vec{e}_3 + \vec{e}_4, \vec{e}_1 + \vec{e}_4\big).\]
Тогда $X_{6}(\vec{z}_4) = X_{6}(2\cdot\frac{1}{2}(\vec{z}_{1} + \vec{z}_{3} + \vec{z}_{4} - \vec{z}_{2}) - \vec{z}_1 + \vec{z}_2 - \vec{z}_3) = \vec{e}_1 - \vec{e}_2 + \vec{e}_3$ и $X_{6}GX_{6}^{-1} = G_{6}$, так как по лемме \ref{main_lem}
\[X_{6}GX_{6}^{-1}\big(\vec{e}_1 + \vec{e}_2 + \vec{e}_4, \vec{e}_1, \vec{e}_1 - \vec{e}_3 + \vec{e}_4, \vec{e}_1 - \vec{e}_2 + \vec{e}_3\big) = \]
\[\big(\vec{e}_1, \vec{e}_1 - \vec{e}_3 + \vec{e}_4, \vec{e}_1 - \vec{e}_2 + \vec{e}_3, \vec{e}_1 + \vec{e}_2 + \vec{e}_4\big).\]
Стало быть, $X_{6}\big(\cf(l_1,l_2,l_3,l_4)\big) \in  \mathbf{CF}_{6}$, то есть $\cf(l_1,l_2,l_3,l_4)\in\overline{\mathbf{CF}}_6$.

Пусть выполняется утверждение \textup{(7)} леммы \ref{main_lem}. Рассмотрим такой оператор $X_{7} \in \Gl_4(\Z)$, что
\[X_{7}\big(\vec{z}_1, \vec{z}_{2}, \vec{z}_{3}, \frac{1}{2}(\vec{z}_{1}+\vec{z}_{2}) + \frac{1}{4}(\vec{z}_{1}+\vec{z}_{4} - \vec{z}_{3} - \vec{z}_{2})\big) =\]
\[\big(\vec{e}_1+\vec{e}_2-\vec{e}_3+2\vec{e}_4, \vec{e}_1-\vec{e}_2+2\vec{e}_3, \vec{e}_1 + 2\vec{e}_2 + 2\vec{e}_4, \vec{e}_1+\vec{e}_4\big).\]
Тогда $X_{7}(\vec{z}_4) = X_{7}(4(\frac{3\vec{z}_1+\vec{z}_2-\vec{z}_3+\vec{z}_4}{4})-3\vec{z}_1-\vec{z}_2+\vec{z}_3) = \vec{e}_1  + \vec{e}_3$ и $X_{7}GX_{7}^{-1} = G_{7}$, так как по лемме \ref{main_lem}
\[X_{7}GX_{7}^{-1}\big(\vec{e}_1+\vec{e}_2-\vec{e}_3+2\vec{e}_4, \vec{e}_1-\vec{e}_2+2\vec{e}_3, \vec{e}_1 + 2\vec{e}_2 + 2\vec{e}_4, \vec{e}_1  + \vec{e}_3\big) = \]
\[\big(\vec{e}_1-\vec{e}_2+2\vec{e}_3, \vec{e}_1 + 2\vec{e}_2 + 2\vec{e}_4, \vec{e}_1  + \vec{e}_3, \vec{e}_1+\vec{e}_2-\vec{e}_3+2\vec{e}_4\big).\]
Стало быть, $X_{7}\big(\cf(l_1,l_2,l_3,l_4)\big) \in  \mathbf{CF}_{7}$, то есть $\cf(l_1,l_2,l_3,l_4)\in\overline{\mathbf{CF}}_7$.
\end{proof}

\section*{Благодарности}
Автор выражает отдельную благодарность О.Н. Герману за постоянный интерес к данной работе и многочисленные, исключительно важные обсуждения результатов.

Автор является победителем конкурса «Junior Leader» Фонда развития теоретической физики и математики «БАЗИС» и хотел бы поблагодарить жюри и спонсоров конкурса.

\end{document}